\documentclass[reqno,english]{amsart}

\usepackage{amsmath,amsfonts,amssymb,graphicx,amsthm,enumerate,url}
\usepackage[noadjust]{cite}
\usepackage{stmaryrd}
\usepackage{comment,paralist,mathrsfs}
\usepackage{mathrsfs,booktabs,tabularx}
\usepackage{xifthen,xcolor,tikz,setspace}
\usetikzlibrary{decorations.pathmorphing,patterns,shapes,calc,decorations}
\usetikzlibrary{decorations.pathreplacing}
\usepackage{mathtools}
\usepackage[small]{caption}
\usepackage{subcaption}

\usepackage[colorinlistoftodos]{todonotes}
\usepackage[colorlinks=true]{hyperref}

\numberwithin{equation}{section}

\renewcommand{\epsilon}{\varepsilon}

\usepackage{color}

\newtheorem{theorem}{Theorem}[section]
\newtheorem*{theorem*}{Theorem}
\newtheorem{lemma}[theorem]{Lemma}

\newtheorem{claim}[theorem]{Claim}

\newtheorem*{observation*}{Observation}

\newtheorem{corollary}[theorem]{Corollary}

\newtheorem{condition}[theorem]{Condition}

\theoremstyle{definition}{
\newtheorem{example}[theorem]{Example}
\newtheorem{definition}[theorem]{Definition}
\newtheorem*{definition*}{Definition}

}

\newcommand{\E}{\mathbb E}
\newcommand{\N}{\mathbb N}
\renewcommand{\P}{\mathbb P}

\newcommand{\bbT}{\mathbb T}

\newcommand{\R}{\mathbb R}
\newcommand{\C}{\mathbb C}

\newcommand{\cD}{\mathcal D}
\newcommand{\cG}{\mathcal G}

\newcommand{\cH}{\mathcal H}

\newcommand{\ii}{\mathrm{i}}

\newcommand{\p}{{\bf p}}
\newcommand{\q}{{\bf q}}
\newcommand{\tr}{{\mathrm {tr}}}
\newcommand{\Id}{\mathrm{Id}}
\newcommand{\Bin}{\mathrm{Bin}}

\DeclareMathOperator{\dist}{dist}

\begin{document}

\title{On the $d$-rigidity phase transition in random graphs}

\author{Yuval Peled}
\address{Einstein Institute of Mathematics, The Hebrew University of Jerusalem, Jerusalem, Israel}
\email{yuval.peled@mail.huji.ac.il}
\thanks{Y.P. was partially supported by the Israel Science Foundation grant ISF-3464/24.}
\maketitle

\begin{abstract}
We study generic \(d\)-dimensional rigidity in sparse random graphs.
Our main result is that for every $d\ge 2$, the Erd\H{o}s--R\'enyi random graph
\(G\sim G(n,c/n)\) undergoes a \(d\)-rigidity phase transition at the known, explicit, \(d\)-orientability threshold \(c_d\):

If \(c<c_d\), then $G$ is asymptotically almost surely (a.a.s.) independent in the
generic \(d\)-rigidity matroid. Moreover, in this regime \(G\) has no
linear-size rigidity components: it contains no induced \(d\)-rigid
subgraphs with more than \(3\) vertices, and the largest clique in its
\(d\)-rigidity closure has size at most \(o(\sqrt n)\). 

If \(c>c_d\),
then $G$ is a.a.s. not independent in the generic \(d\)-rigidity matroid, and we give a sharp asymptotic estimate for its rank. In addition, the \(d\)-rigidity closure of $G$ has a giant clique of linear
size, which 
contains all but at most \(o(n)\) vertices of the \(((d+1)+d)\)-core of the graph.

More generally, we compute, up to a \(1+o(1)\) factor, the generic
\(d\)-rigidity rank of random graphs with a given degree distribution. For
example, we show that the uniform \(n\)-vertex \(k\)-regular graph a.a.s. has rank $\min(k/2,d)n+o(n).$
Our approach  is to estimate the rigidity rank of a random graph
from its Galton--Watson local weak limit, using a parameter that we call
{\em local flexibility}.

\end{abstract}

\section{Introduction}\label{sec:intro}
The systematic study of random graphs was initiated by Erd\H{o}s and
R\'enyi in their seminal papers
\cite{ErdosRenyi1959,ErdosRenyi1960}, and has played an immense role in
modern discrete mathematics, computer science, and related areas
\cite{BollobasRandomGraphs,JansonLuczakRucinskiRandomGraphs,FriezeBook}.
This work studies random graphs from the point of view of generic
rigidity in \(\R^d\), a major subject in modern discrete geometry since the works of Laman \cite{Laman1970} and Asimow--Roth
\cite{AsimowRoth1978,AsimowRoth1979}. While \(1\)- and \(2\)-dimensional
generic rigidity are well understood, characterizing generically rigid
graphs in dimensions \(d\ge 3\) remains a major open problem. Our work focuses on this regime of higher-dimensional rigidity of sparse random graphs.

A \(d\)-dimensional framework is a pair \((G,\p)\), where
\(G=(V,E)\) is a finite simple graph and \(\p:V\to\R^d\). The framework is called rigid if the only continuous motions
of the vertices in \(\R^d\) starting at \(\p\) and preserving the edge lengths are
induced by isometries of \(\R^d\). Asimow and Roth showed that if \(\p\) is generic then the rigidity of the framework \((G,\p)\) depends only
on the underlying graph \(G\), and is equivalent to the rigidity matrix \(R(G,\p)\) attaining a maximal rank, see \cite{GraverServatiusServatius1993}.

One of the classical results on the \(G(n,p)\) random graph model, in which every edge is included
independently with probability \(p\), is that a.a.s., i.e., with probability
tending to one as \(n\to\infty\), the graph is connected if and only if it
has no isolated vertices. Since graph connectivity is equivalent to
\(1\)-dimensional rigidity, it is natural to ask for higher-dimensional
analogs: a.a.s., \(d\)-rigidity occurs if and only if the
minimum degree of the graph is at least \(d\). Indeed, this was first proved for $d=2$ by Jackson, Servatius and Servatius \cite{JSS-planeThreshold}, and extended to all fixed dimensions
\(d\ge 3\) by Lew, Nevo, Raz and the author \cite{LewNevoPeledRazSharpThreshold} (see also \cite{KiralyTheranCoherence,JordanTanigawaRigidityRandomSubgraphs}). Rigidity of \(G(n,p)\) has also been studied in other
settings where the minimum degree is not the only obstruction to its rigidity,
such as when \(d=d(n)\to\infty\)~\cite{KrivelevichLewMichaeliRigidPartitions,KrivelevichCombinatorialSufficientConditions,PeledPelegHighDimensional,JordanLiuVillanyiDegreeSum}, and for non-generic embeddings~\cite{BenjaminiTzalikDetermining,GiraoIllingworthMichelPowierskiScottReconstructing,MontgomeryNenadovPortierSzaboGlobalRigidityR,ClinchHaslegraveHuynhNixonNAC}.

Another fundamental discovery in random graph theory is the phase transition in \(G(n,p)\) at \(p=1/n\): 
for \(c<1\), the random graph \(G(n,c/n)\) a.a.s. consists of connected components of size \(O(\log n)\) and only \(O(1)\) cycles,
while for \(c>1\), it a.a.s. contains a unique giant component of linear size and its cycle space has linear dimension.
This was extended to \(2\)-dimensional rigidity by 
Kasiviswanathan, Moore and Theran~\cite{KasiviswanathanMooreTheranRigidityTransition} (see also~\cite{TheranRigidComponentsRandomGraphs}). This paper establishes an analogous phase transition for \(d\)-dimensional rigidity of $G(n,p)$ for every fixed \(d\ge 3\).

\medskip

The infinitesimal point of view of Asimow and Roth leads to several interesting notions in
rigidity that are needed to state our results. 
The \(|E|\times d|V|\) rigidity matrix \(R(G,\p)\) is the differential at \(\p\) of the map $\R^{d|V|}\to\R^{|E|}$ given by
\(
\q\mapsto \left(\frac12\|\q(u)-\q(v)\|_2^2\right)_{uv\in E}.
\)
We assume that $|V|=n>d$ and $\p$ is generic. Then, the isometries of $\R^d$ induce a $\binom{d+1}{2}$-dimensional subspace of the kernel of $R(G,\p)$.
\(G\) is called {\em \(d\)-rigid} if 
the rank $\mathrm{rk}_d(G)$ of $R(G,\p)$ is equal to
$dn-\binom{d+1}{2}$, and 
{\em $d$-independent} if $\mathrm{rk}_d(G)=|E(G)|$.

Let \(\tilde G\) denote the \((d+1)\)-core of \(G\), i.e., the maximal
subgraph of minimum degree at least \(d+1\). Denote by
$n_{d+1}=|V(\tilde G)|$ and $m_{d+1}=|E(\tilde G)|$.
Clearly,
\[
    \mathrm{rk}_d(\tilde G)
    \le \min\left(m_{d+1},\,dn_{d+1}\right).
\]
In addition, removing a vertex \(v\) of degree \(d_v\le d\) reduces
the rank by exactly \(d_v\). Therefore, by denoting $m=|E(G)|$, we have
\begin{equation}\label{eq:core-rank-upper-bound}
    \mathrm{rk}_d(G) = m-m_{d+1}+\mathrm{rk}_d(\tilde G)
    \le
    m+\min\left(0,\,dn_{d+1}-m_{d+1}\right).
\end{equation}
A main theme of this work is that \eqref{eq:core-rank-upper-bound} is asymptotically sharp in various models of sparse random graphs. The \((d+1)\)-core of \(G(n,c/n)\) has been studied
extensively, and its asymptotic size and density are well understood
\cite{PittelSpencerWormaldCore,JansonSimplekcore}. In particular,
let \(\hat p=\hat p(c)\) denote the largest solution in
\([0,1]\) of $p=\P(\operatorname{Poi}(cp)\ge d).$
There exists a critical value $\gamma_d$ such that $\hat p>0$ if and only if $c\ge \gamma_d$, and the $(d+1)$-core of $G(n,c/n)$ is known to emerge at $c=\gamma_d$. 
Moreover, a.a.s.,
\[
    n_{d+1}
    =
    \left(\P(\operatorname{Poi}(c\hat p)\ge d+1)+o(1)\right)n,
    \qquad
    m_{d+1}
    =
    \left(\frac{c\hat p^2}{2}+o(1)\right)n .
\]

The threshold \(c_d\) at which the \((d+1)\)-core
first has edge density \(d\) was shown to be
critical for \(d\)-orientability of \(G(n,p)\), the property
that the edges can be oriented so that every vertex has in-degree at most
\(d\)~\cite{CainSandersWormald,FernholzRamachandran}.

\begin{theorem}\label{thm:mainGnp_rank}
    Fix $d\ge 2$, and let $G\sim G(n,c/n)$, where $c>0$ is fixed. Then a.a.s. the following hold:
    \begin{enumerate}
        \item If \(c<c_d\) then \(G\) is \(d\)-independent.
        \item If \(c>c_d\) then \(G\) is not \(d\)-independent and
        \[
        \mathrm{rk}_d(G) = (1+o(1))\left(\frac c2 + d\,\P\big(\operatorname{Poi}(c\hat p)\ge d+1\big)- \frac{c\hat p^2}{2}\right)n\,.
        \]
    \end{enumerate}
    \end{theorem}
    In other words, the upper bound given by \eqref{eq:core-rank-upper-bound} is attained asymptotically.

\begin{figure}[h!]
    \centering
    \includegraphics[width=0.85\textwidth]{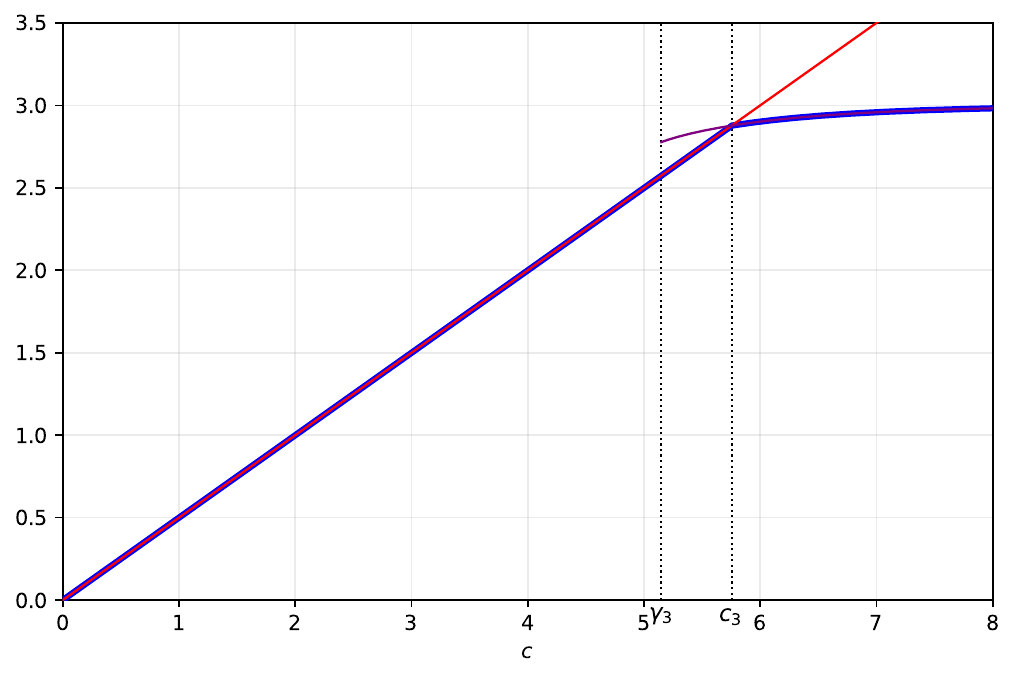}
    \caption{
    Illustration of Theorem~\ref{thm:mainGnp_rank} for \(d=3\).
    The blue curve is the asymptotic value of \(\mathrm{rk}_3(G)/n\).
    The red curve is the trivial edge-density bound \(c/2\), and the purple curve is the second bound in \eqref{eq:core-rank-upper-bound} with the parameters of the $4$-core.
    The vertical lines mark the thresholds
    $\gamma_3 \approx 5.1494,$ and
$c_3 \approx 5.7549.$
    }
    \label{fig:d3-rank-plot}
\end{figure}
For $c<\gamma_d$, the first item of Theorem \ref{thm:mainGnp_rank} follows directly from the fact that the $(d+1)$-core of $G$ is a.a.s. empty. The main contribution of this item is that $G$ is $d$-independent if $\gamma_d\le c<c_d$ --- when the $(d+1)$-core exists but its edge density is smaller than $d$. For illustration,
\[
\begin{array}{|c|c|c|c|c|c|}
\hline
d & 2 & 3 & 4 & 5 & d\to\infty \\ \hline
\gamma_d
& 3.3509 & 5.1494 & 6.7993 & 8.3653
& d+\sqrt{d\log d}+O(\sqrt d) \\ \hline
c_d
& 3.5880 & 5.7549 & 7.8430 & 9.8955
& 2d-\dfrac{4+o(1)}{\sqrt{2\pi d}}\left(\dfrac{2}{e}\right)^d \\\hline
\end{array}
\]

Related phenomena have been studied in physics and chemistry under the names rigidity percolation and topological constraint theory, particularly in network glasses and other disordered materials, where the percolation of rigid clusters is associated with the onset of macroscopic mechanical rigidity \cite{Phillips1979,Thorpe1983,Thorpe1985,JacobsThorpe1995}.

In recent years, fascinating connections between maximum likelihood thresholds (MLT) and rigidity
theory have been discovered and developed \cite{Uhler2012,GrossSullivant2018,BernsteinDewarGortlerNixonSitharamTheranMLT}.
The minimum dimension $d$ in which a graph is $d$-independent plays an important role
in this area. Theorem~\ref{thm:mainGnp_rank} identifies this dimension for
\(G(n,c/n)\), making progress toward a conjecture from
\cite{BernsteinDewarGortlerNixonSitharamTheranMLT} concerning its MLT (see Section \ref{sec:discussion}).

\medskip

The linear matroid on \(E(K_n)\) given by the rigidity matrix is
called the generic \(d\)-rigidity matroid. We denote by
\(\operatorname{cl}_d(G)\) the closure of \(G\) in this matroid, namely the graph whose edges consist of all \(uv\in E(K_n)\) whose addition to \(G\) does not increase
the \(d\)-rigidity rank. The next theorem shows that the phase transition
at \(p=c_d/n\) is also reflected in the emergence of a giant clique in the
\(d\)-rigidity closure, which spans almost all the $((d+1)+d)$-core of the graph. This is the subgraph obtained from the
\((d+1)\)-core by repeatedly adding vertices that have at least \(d\)
neighbors in the subgraph already constructed.

\begin{theorem}\label{thm:mainGnp_closure}
    Fix \(d\ge 2\), and let \(G\sim G(n,c/n)\), where \(c>0\) is fixed.
    Then a.a.s. the following hold:
    \begin{enumerate}
        \item If \(c<c_d\), then
        \begin{enumerate}
            \item \(G\) contains no induced \(d\)-rigid subgraph with more than \(3\) vertices.
            \item \(|\operatorname{cl}_d(G)\setminus E(G)|=o(n)\), and the largest clique in \(\operatorname{cl}_d(G)\) has size \(o(\sqrt n)\).
        \end{enumerate}
        \item If \(c>c_d\), then the largest clique in \(\operatorname{cl}_d(G)\) has size \((\hat p+o(1))n\), and it contains
         all but at most \(o(n)\) vertices of the \(((d+1)+d)\)-core of $G$.
    \end{enumerate}
\end{theorem}

\begin{figure}[h!]
    \centering
    \includegraphics[width=0.75\textwidth]{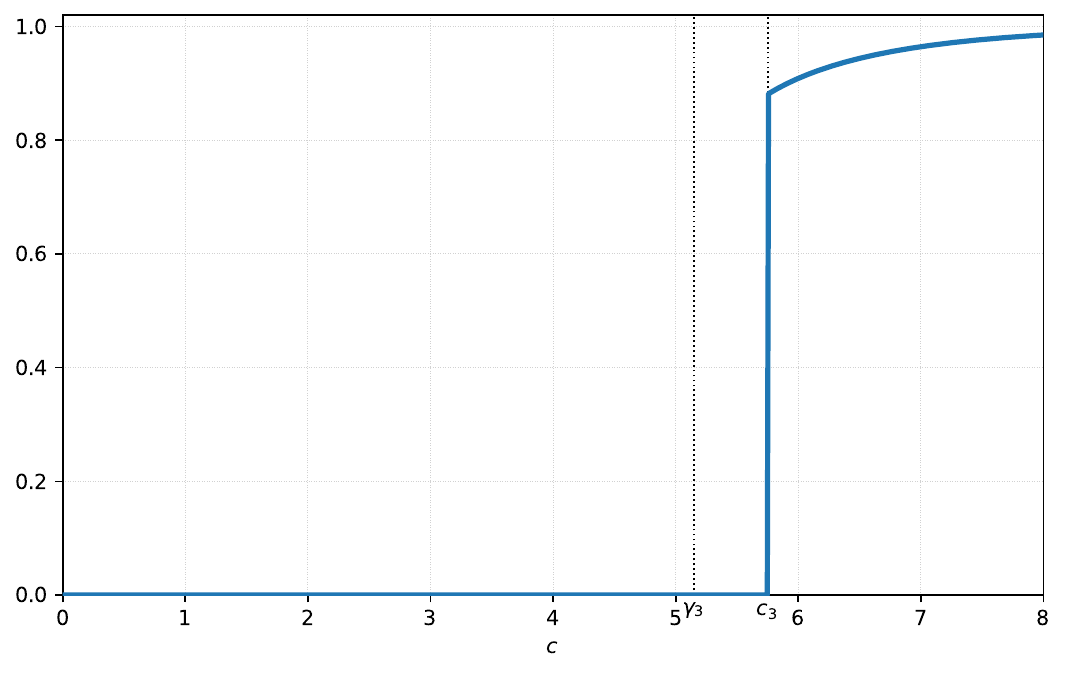}
    \caption{
    Illustration of Theorem~\ref{thm:mainGnp_closure} for \(d=3\).
    The blue curve is the asymptotic density of the largest clique in the $3$-rigidity closure of $G$.
    }
    \label{fig:d3-rank-plot}
\end{figure}

Note that a local heuristic suggests that the \(((d+1)+d)\)-core of
\(G(n,c/n)\) a.a.s. contains \((\hat p+o(1))n\) vertices. While this was
proved only for \(d=2\) in \cite{BarreLelargeMitscheSliders}, it seems
likely that the same arguments work for all $d$.

Item~\((1a)\) was conjectured in \cite{LewNevoPeledRazSharpThreshold}, as
was a stronger version of item~\((2)\): that the \emph{entire}
\(((d+1)+d)\)-core \emph{induces} a \(d\)-rigid graph.
In addition, Theorems~\ref{thm:mainGnp_rank} and
\ref{thm:mainGnp_closure} were proved for \(d=2\) in
\cite{KasiviswanathanMooreTheranRigidityTransition}. In this case it is also known that a.a.s.
\(\operatorname{cl}_2(G)=G\) when \(c<c_2\), and we conjecture that the
same holds for every \(d\ge 2\). In fact, if $d=2$ then the
subgraph induced by any inclusion-maximal clique in the rigidity closure is
itself \(2\)-rigid. This fails for \(d\ge 3\), and therefore
Theorem~\ref{thm:mainGnp_closure} does not imply the existence of a giant
\(d\)-rigid induced subgraph when \(c>c_d\). However, for sufficiently large
\(c>0\), such a giant \(d\)-rigid induced subgraph was shown to exist a.a.s.
by Krivelevich, Lew and Michaeli~\cite{KrivelevichCombinatorialSufficientConditions}.

\medskip

Another well-studied model is the uniform \(n\)-vertex \(k\)-regular graph
\(G_{n,k}\). It is well-known that for a fixed $k\ge 3$, $G_{n,k}$ is a.a.s. of essentially high-girth as $n\to\infty$. That is, for any fixed $g$, only $o(n)$ vertices are contained in a cycle of length at most $g$.
Clearly, a \(k\)-regular graph cannot be \(d\)-rigid when
\(k<2d\), since it has too few edges. It is conjectured that for every
\(d\ge 2\) and \(k\ge 2d\), the random graph \(G_{n,k}\) is a.a.s.
\(d\)-rigid \cite{KrivelevichLewMichaeliRigidPartitions}. This was proved for \(d=2\) in \cite{JSS-planeThreshold}, and for $k\ge 501d$, for every $d\ge 3$, in \cite{KrivelevichCombinatorialSufficientConditions}.
For \(d\ge 3\) and small $k$, the question remains open, and we prove an
asymptotic version of this conjecture. In fact, our theorem applies to any sequence of \(k\)-regular graphs of essentially high-girth.

\begin{theorem}\label{thm:main_high_girth_regular}
    Fix \(d\ge 2\) and \(k\ge 3\). Let \(G\) be an $n$-vertex 
    \(k\)-regular graph of essentially high-girth. Then, as $n\to\infty$,
    \[
        \mathrm{rk}_d(G)
        =
        \left(\min\left\{\frac{k}{2},d\right\}+o(1)\right)n .
    \]
\end{theorem}

Note that Theorem~\ref{thm:main_high_girth_regular} gives another example
in which the upper bound \eqref{eq:core-rank-upper-bound} is asymptotically
tight, since the \((d+1)\)-core of a \(k\)-regular graph is the graph itself
whenever \(k\ge d+1\). In fact, Theorem~\ref{thm:rd_degreeX} below shows that,
under suitable technical assumptions on a distribution \(X\) on
\(\mathbb N_{\ge 0}\), a uniform random graph whose degree distribution is
sampled from \(X\) a.a.s. has \(d\)-rigidity rank asymptotically attaining
the upper bound in \eqref{eq:core-rank-upper-bound}. Thus
Theorem~\ref{thm:mainGnp_rank} corresponds closely to the special case
\(X=\operatorname{Poi}(c)\), while Theorem~\ref{thm:main_high_girth_regular}
 to the case where \(X\equiv k\) is the point mass at \(k\).

\medskip

Our proof adapts ideas from the work of Bordenave, Lelarge and Salez on the
rank of sparse random graphs \cite{BordenaveLelargeSalezDiluted}, and from
the work of Linial and the author on the homology of random simplicial
complexes \cite{LinialPeledHomology}. The common theme is that the sparse
random graphs under consideration converge locally, in the sense of
Benjamini--Schramm \cite{BenjaminiSchramm2001}, to a limiting rooted tree.
For example, \(G(n,c/n)\) converges locally to a Galton--Watson tree with
offspring distribution \(\operatorname{Poi}(c)\).

We define a parameter \(\phi_{G,\p}(v)\in[0,d]\), called the local
flexibility of a vertex \(v\) in a framework \((G,\p)\), which quantifies
the freedom of motion of \(v\) in the framework. The sum of
\(\phi_{G,\p}(v)\) over all vertices is equal to the dimension of the
kernel of \(R=R(G,\p)\), and since this parameter can be read from the spectral
measure of \(R^*R\), it can be approximated by the local flexibility of the
root in the limiting tree.

To apply this approach, we extend basic notions in infinitesimal rigidity theory to infinite frameworks via the theory of self-adjoint unbounded operators. We then compute the expected local flexibility of Galton--Watson trees using their recursive structure, yielding the main theorems.

\medskip

The purely infinite aspects of our work may also be of independent interest. For
example, we prove an infinite analogue of the conjecture on the rigidity of
random \(k\)-regular graphs. Say that an infinite framework \((G,\p)\) is
strictly \(\ell_2\) \(d\)-rigid if it admits no nonzero infinitesimal motions
in \(\ell_2(V;\R^d)\). That is, there is no nonzero
\(q\in\ell_2(V;\R^d)\) such that for every edge \(uv\in E(G)\),
\[
    \langle q(u)-q(v),\,\p(u)-\p(v)\rangle=0.
\]
This definition is not vacuous: in contrast with finite frameworks, the trivial motions need not lie in $\ell_2$. See Section~\ref{sec:inf_fram} for the precise definition, its motivation, and further discussion.

\begin{theorem}\label{thm:Tk}
    Let \(d\ge 2\) and \(k\ge 2\). Let \(\bbT_k\) be the infinite
    \(k\)-regular tree, and let \(\nu\) be the uniform distribution on
    \([0,1]^d\). Sample \(\p:V(\bbT_k)\to\R^d\) by drawing each
    \(\p(v)\) independently from \(\nu\). Then,
    \[
    \P\big((\bbT_k,\p)\text{ is strictly \(\ell_2\) \(d\)-rigid}\big)
    =
    \begin{cases}
        1, & k\ge 2d,\\
        0, & k<2d.
    \end{cases}
\]
\end{theorem}

\medskip

The remainder of the paper is organized as follows. In
Section~\ref{sec:preliminaries} we give the necessary background material.
In Section~\ref{sec:inf_fram} we define and study rigidity for infinite
frameworks. In Section~\ref{sec:local_flex} we introduce local flexibility
and prove its upper semi-continuity under Benjamini--Schramm convergence.
The main technical work is carried out in Section~\ref{sec:trees}: in
Subsection~\ref{sec:recursive} we establish a recursive formula for the
local flexibility of trees, and in Subsection~\ref{sec:GW} we solve this
recursion for Galton--Watson trees. In Section~\ref{sec:applications} we
apply this solution to prove the main theorems of the paper. We conclude
with open problems in Section~\ref{sec:discussion}.

\section{Preliminaries}\label{sec:preliminaries}
\subsection{Infinitesimal rigidity theory}
Throughout this paper, we fix a dimension $d\ge 2$. Let $G$ be an $n$-vertex graph and $\p:V(G)\to\R^d$ a generic embedding. We view every $\p(v)$ as a column vector in $\R^d$, 
and denote $\p_{uv}=\p(u)-\p(v)$ for two distinct vertices $u$ and $v$. 
The rigidity matrix $R = R(G,\p)$ of the $d$-framework $(G,\p)$ is the $|E| \times d|V|$ matrix whose rows are indexed by edges $uv \in E$ and whose columns are grouped into $|V|$ blocks of size $d$, one block for each vertex. 
For an edge $uv \in E$, the row corresponding to $uv$ has $\p_{uv}^*$ in the block corresponding to $u$, $\p_{vu}^*$ in the block corresponding to $v$, and zero elsewhere.

Throughout this paper it will be more convenient to work with the {\em rigidity Laplacian} $L=L(G,\p)=R^TR$ (defined more explicitly below), as it is a self-adjoint operator.
We denote by $Z=Z(G,\p)$ the kernel of $R(G,\p)$; equivalently, $Z$ is the kernel of the rigidity Laplacian $L(G,\p)$. 
A \(dn\)-dimensional vector \(q\) in \(Z\) can be viewed as assigning an
infinitesimal velocity vector \(q(v)\in\R^d\) to each vertex, describing
an infinitesimal motion of the vertices at \(\p\) under which all edge
lengths are preserved to first order. If \(\p\) is generic then the
isometries of \(\R^d\) induce a \(\binom{d+1}{2}\)-dimensional subspace of
\(Z\), which we refer to as the {\em trivial motions}. More concretely, the
subspace of trivial motions is parametrized by a \(d\times d\)
skew-symmetric matrix \(A\) and a vector \(t\in\R^d\), yielding a vector
\(q\in Z\) defined by
\[
q(v)=A\p(v)+t,\qquad v\in V.
\]
$G$ is $d$-rigid iff $Z$ consists only of trivial motions. 

Another matrix that will play a role in this work is the orthogonal
projection matrix \(P:\R^{dn}\to Z\). Intuitively, suppose that each vertex
has unit mass and that the edges are massless bars. If \(f\in \R^{dn}\) describes a vector of instantaneous external forces,
where \(f(v)\in\R^d\) is the force applied to the vertex \(v\), then
\(Pf\) is the component of this force in the space of admissible
infinitesimal motions; equivalently, for unit masses, it is the resulting
instantaneous acceleration after imposing the bar constraints.
One may consider this locally. Let \(P_{v,v}\) denote the
\(d\times d\) diagonal block of \(P\) corresponding to the vertex \(v\).
If a force \(f(v)\) is applied only to \(v\), then the resulting
instantaneous acceleration of \(v\) is
$P_{v,v} f(v).$ Thus, the trace of $P_{v,v}$ quantifies the degrees of freedom that $v$ has in the framework, and we will refer to it as the local flexibility of $v$; see Section \ref{sec:local_flex} below for a formal definition and further discussion.

\subsection{Benjamini--Schramm convergence}
A main idea in this work is to estimate the rigidity rank of a finite graph by its local structure. 
For this, we work with the local weak topology of rooted graphs introduced in the influential paper 
of Benjamini and Schramm \cite{BenjaminiSchramm2001}. 
This point of view is closely related to the objective method of Aldous and Steele
\cite{AldousSteeleObjectiveMethod}, and has been used to study a variety of
parameters of sparse graphs which are governed by local limits. Of particular
relevance to us are results concerning spectral measures and kernels of
operators on locally convergent graphs; see, for example,
\cite{AbertThomViragSpectralMeasure,BordenaveLelarge2010,BordenaveLelargeSalezDiluted}.

A rooted graph $(G,o)$ is a graph with a root vertex $o\in V(G)$. Let $\cG_\bullet$ denote the set of isomorphism classes of connected locally finite rooted graphs $(G,o)$, 
where isomorphisms are required to preserve the root. 
For $(G,o)\in\cG_\bullet$ and an integer $r\ge 0$, we write $B_G(o,r)$ for the rooted induced subgraph spanned by all vertices 
at graph distance at most $r$ from $o$; thus $B_G(o,r)\in\cG_\bullet$ as well. We metrize $\cG_\bullet$ by defining 
the distance between $(G_1,o_1)$ and $(G_2,o_2)$ to be $2^{-r}$, where $r\ge 0$ is the largest radius for which the 
rooted balls $B_{G_1}(o_1,r)$ and $B_{G_2}(o_2,r)$ are root-preserving isomorphic. The induced metric is the 
{\em local topology} on $\cG_\bullet$; with this topology, $\cG_\bullet$ is a Polish space, so the standard notions of tightness and convergence in distribution apply.
We say that a (possibly random) rooted graph $(G,o)$ is the local weak (Benjamini--Schramm) limit of a sequence of (possibly random) finite graphs $G_n$ if $(G_n,o_n)$ 
converges in distribution to $(G,o)$ in $\cG_\bullet$, where $o_n$ is chosen uniformly from $V(G_n)$.

We mention two classical examples of local weak convergence from graph theory that appear in this paper. First, $G(n,c/n)$, where $c>0$ is fixed, converges locally to a 
Galton--Watson branching process with $\operatorname{Poi}(c)$ progeny distribution (defined precisely in the next subsection) as $n\to\infty$. 
Second, the uniform random $k$-regular $n$-vertex graph, where $k\ge 3$ is fixed, converges locally to the infinite $k$-regular tree $\bbT_k$ as $n\to\infty$. 
Both these examples can be viewed as special cases of local weak convergence of uniform random graphs with a given degree distribution 
to a corresponding unimodular Galton--Watson process.

\subsection{Galton--Watson processes, size-biasing and thinning}
The graphs we study here converge locally to Galton--Watson processes. 
A Galton--Watson branching process is a random rooted tree generated recursively by letting each vertex produce an independent number of children.  
Let $X,Y$ be probability distributions on $\N_{\ge 0}$. Denote by $\mathrm{GW}(X,Y)$ the distribution on random, locally finite rooted trees $T$ generated 
by the Galton--Watson branching process in which the root $o$ has offspring distribution $X$, while every other vertex has offspring distribution $Y$. 

Galton--Watson processes that arise as local limits of finite graphs are unimodular, which is equivalent to the condition that $Y$ is the size-biasing of $X$:
\begin{equation}\label{eq:size_bias}
Y(m-1)=\frac{m\,X(m)}{\E[X]}\,,\quad\forall m\ge 1\,.
\end{equation}
Here we assume $\E[X]<\infty$.

Two main examples that we consider in this paper are:
\begin{enumerate}
    \item \(X=Y=\operatorname{Poi}(c)\), for some fixed \(c>0\), which is the
    local weak limit of \(G(n,c/n)\);
    \item the \(k\)-regular tree \(\bbT_k\), i.e.,
    \(X(k)=Y(k-1)=1\), which is the local weak limit of the random
    \(k\)-regular graph.
\end{enumerate}

An easy and fundamental fact about size-biasing is that for every $f:\N_{\ge 1}\to\R$ there holds,
\begin{equation}\label{eq:sizeBiasF}
\E[Xf(X-1)] = \E[X]\E[f(Y)]\,.
\end{equation}

In addition, we will use the operation of thinning, and its relation to size-biasing. Given $0\le p \le 1$, we define a coupled pair $(X_p,X_{1-p})$ by 
first sampling $X$, and then, conditional on $X$, sampling $X_p\sim \Bin(X,p)$ and setting $X_{1-p}:=X-X_p$. We can also view this pair as coupled with $X$.
The following claims are useful in our argument.
\begin{claim}\label{clm:sizeBias_new_claim}
Suppose that $Y$ is the size-biasing of $X$ and $0\le p\le 1$. Then,
\begin{enumerate}
\item $Y_p$ is the size-biasing of $X_p$.
\item For every positive integer $d>0$ there holds $$(1-p)\E[X]\P(Y_p<d)=\E[\mathbf 1_{X_p<d}\cdot X_{1-p}]\,.$$
\item Suppose that $t\sim Y_{1-p}$ and that conditioned on $t$, $h\sim X_p | \{X_{1-p}=t+1\}$. Then, the pair $(h,t)$ is $(Y_p,Y_{1-p})$-distributed.
\end{enumerate}
\end{claim}
\begin{proof}
First, for every $t \ge 0$ there holds 
\begin{align}
\label{eq:yp_sizebias_of_xp}
(t+1)\,X_p(t+1)
&= \E\!\left[\mathbf 1_{X_p=t+1}\cdot X_p\right] \notag\\
&= \E\!\left[\E\!\left[\mathbf 1_{X_p=t+1}\cdot X_p \mid X\right]\right] \notag\\
&= \E\!\left[X\cdot p \cdot \P(\Bin(X-1,p)=t )\right] \notag\\
&= p\,\E[X]\cdot \E[\P(\Bin(Y,p)=t)] \notag\\
&= \E[X_p]\cdot Y_p(t)\,, \notag
\end{align}
as claimed. The third equality follows from linearity of expectation, since given $X$, $\mathbf 1_{X_p=t+1}\cdot X_p$ is a sum of $X$ Bernoullis with probability $p\cdot\P(\Bin(X-1,p)=t).$ In addition, the fourth equality is derived by \eqref{eq:sizeBiasF} using $f(k)=\P(\Bin(k,p)=t).$

Similarly, for the second item, by linearity of expectation,
\[
\E[\mathbf 1_{X_p<d}\cdot X_{1-p}\mid X] = X\cdot (1-p)\cdot\P(\Bin(X-1,p)<d)\,,
\]
and by taking expectation over $X$ we obtain
\begin{align*}
\E\!\left[\mathbf 1_{X_p<d}\,X_{1-p}\right]
&=(1-p)\,\E\!\left[X\cdot\P\!\left(\Bin(X-1,p)<d \right)\right] \\
&=(1-p)\,\E[X]\cdot\E[\P\!\left(\Bin(Y,p)<d\right)]
\;=\;(1-p)\,\E[X]\,\P\!\left(Y_p<d\right).
\end{align*}
The transition to the second line is derived by \eqref{eq:sizeBiasF} using $f(k)=\P(\Bin(k,p)<d)$.

For the third item, note that for every $t,h$ it holds that
\begin{align*}
\P(X_p=h\mid X_{1-p}=t+1) &=~ \frac{X(h+t+1)\P(\Bin(h+t+1,p)=h)}{X_{1-p}(t+1)}\\[0.6em]
&=~ \frac{(1-p)(h+t+1)X(h+t+1)\P(\Bin(h+t,p)=h)}{(t+1)X_{1-p}(t+1)} \\[0.6em]
&=~ \frac{(1-p)\E[X]Y(h+t)\P(\Bin(h+t,p)=h)}{(t+1)X_{1-p}(t+1)} \\[0.6em]
&=~ \frac{Y(h+t)\P(\Bin(h+t,p)=h)}{Y_{1-p}(t)} \\
&=~ \P(Y_p=h\mid Y_{1-p}=t) \,,
\end{align*}
and the claim follows since we assume $t\sim Y_{1-p}$. Here, the second equality follows from a standard identity of Binomial distributions, the next equality holds since $Y$ is the size-biasing of $X$, and the fourth equality is derived by the first item.
\end{proof}

\section{Infinite frameworks}\label{sec:inf_fram}
An important idea in this work is to study rigidity of infinite graphs, and
infinite trees in particular. Rigidity of infinite frameworks was studied by
Owen and Power \cite{OwenPowerInfiniteFrameworksOperatorTheory} and by
Kitson and Power \cite{KitsonPowerRigidityInfiniteGraphs}. The approach we
take here is closest to the notion of square-summable
infinitesimal rigidity from \cite{OwenPowerInfiniteFrameworksOperatorTheory}.
However, we opt to give a self-contained treatment, partly since our focus is on the spectral-theoretic aspects of this notion.

The infinite graphs we consider are locally finite but may have unbounded degree. Consequently, their rigidity Laplacians may be unbounded operators. Accordingly, we begin by recalling some background material from the theory of unbounded operators~\cite{ReedSimonI}.
Let $\cH$ be a complex Hilbert space. The main issue with unbounded operators is that they cannot be defined on the entire space $\cH$. Hence, an unbounded operator on $\cH$ consists of a domain $D$ which is a linear subspace of $\cH$, and a linear operator $L:D\to \cH$. As we aim to apply spectral methods, we need to work with a self-adjoint operator. However, if $L$ is not closed, that is, if the graph $\{(f,Lf)\mid f\in D\}$ is not closed in $\cH\oplus\cH$, then $L$ cannot be self-adjoint, since the adjoint $L^*$ is always closed. 
Fortunately, in our setting, $D$ is a {\em dense} subspace and $L$ is symmetric. Therefore, $L$ is closable, and we denote its closure by $\bar L$. The domain $D$ is a {\em core} for $\bar L$.
We say that $L$ is essentially self-adjoint if $\bar L$ is self-adjoint.
In this case, the spectral theorem applies to $\bar L$. In particular, for any bounded Borel function $g:\R\to\C$, the operator $g(\bar L)$ is well defined via the functional calculus. Moreover, for every vector $f\in\cH$, there exists a finite Borel measure $\mu_f$ on $\R$, called the spectral measure of $\bar L$ with respect to $f$, such that
\[
\langle f, g(\bar L) f\rangle = \int_\R g(\lambda)\, d\mu_f(\lambda)
\]
for all bounded Borel functions $g$. In particular, if $g=\mathbf 1_{\{0\}}$ then $g(\bar L)$ is the orthogonal projection $P:\cH\to\ker\bar L$ and $\langle f,Pf\rangle=\mu_f(\{0\})$.

All the graphs $G=(V,E)$ we consider
are locally finite on a countable vertex set. Let $\p:V\to\R^d$. Since we use spectral theory, we prefer to work over the complex numbers and consider the Hilbert space
\[
\cH:=\ell_2(V;\C^d)\,.
\]
This makes no difference to us since rigidity theories over $\R$ and $\C$ are identical. Denote by $D\subset\cH$ the subspace of finitely supported vectors. The rigidity Laplacian $L=L(G,\p)$ is given by a (possibly infinite) block matrix
\[
L=(L_{uv})_{u,v\in V}, \qquad 
L_{uv}\in\C^{d\times d},
\]
acting on $f\in D$ by
\[
(L f)(u)=\sum_{v\sim u}L_{uv} f(v),
\]
where the sum is finite for each $u$. The block entries are given by
\begin{equation}\label{eq:Luv}
L_{uv}=
\begin{cases}
\displaystyle \sum_{w\sim u} \p_{uw}\p_{uw}^{*}, & u=v,\\[1ex]
\displaystyle -\,\p_{uv}\p_{uv}^{*}, & \{u,v\}\in E,\\[1ex]
0, & \text{otherwise}\,.
\end{cases}
\end{equation}

The operator $L$ is symmetric and defined on the dense domain $D$, since $L_{uv}=L_{vu}^*$ for all $u,v\in V$. Therefore, $L$ is closable, and we denote its closure by $\bar L$.

\begin{definition}
Let $G$ be a locally finite graph on a countable vertex set and $\p:V(G)\to\R^d$.
The $d$-framework $(G,\p)$ is called \emph{self-adjoint} if $\bar L(G,\p)$ is a self-adjoint operator.
\end{definition}
It is clear that if $G$ is finite, or, more generally, if $G$ has bounded degree and $\p_{uv},~uv\in E$ are bounded, then $(G,\p)$ is self-adjoint. Indeed, in such a case the Rigidity Laplacian is a bounded operator.

Although self-adjointness of a framework is defined via the complex Hilbert
space \(\cH\), if a framework \((G,\p)\) is self-adjoint, we note that the
orthogonal projection
\[
  P=\mathbf 1_{\{0\}}(\bar L):\cH\to \ker(\bar L)
\]
onto the kernel of its rigidity Laplacian can be viewed as a real operator. Indeed, the
rigidity Laplacian has real entries, hence the projection \(P\) preserves the real subspace of $\cH$.
Thus, when applied to real vectors, \(P\) is the orthogonal projection onto
the real kernel of \(\bar L\).

\subsection{$\ell_2$-infinitesimal motions}
The following lemma extends the standard description for the kernel of the rigidity Laplacian from finite graphs to self-adjoint frameworks. In words, the kernel consists of all infinitesimal distance preserving motions that have a finite $\ell_2$-norm.
\begin{lemma}\label{lem:kernel}
Let $(G,\p)$ be a self-adjoint $d$-framework. Then,
\[
\ker \bar L(G,\p) = \{q\in\cH\mid \langle q(u)-q(v),\p_{uv}\rangle=0,~~\forall uv \in E(G)\}\,.
\]
\end{lemma}
\begin{proof}
The dense subspace $D$ of finitely-supported vectors is a {\em core} for the self-adjoint operator $\bar L=\bar L(G,\p)$. That is, for every $q$ in the domain of $\bar L$ there are vectors $f_n \in D$ such that $f_n\to q$ and $Lf_n\to \bar L q$. 

Suppose that $q\in \ker \bar L$, and consider vectors $f_n \in D$ such that $f_n\to q$ and $Lf_n\to 0$. We have that
\[
\sum_{uv\in E}|\langle f_n(u)-f_n(v),\p_{uv}\rangle|^2 = \langle f_n,Lf_n\rangle \to \langle q,\bar Lq\rangle =0\,.
\]
Hence, for every edge $uv \in E$, 
\[
\langle q(u)-q(v),\p_{uv}\rangle=\lim_{n\to \infty}\langle f_n(u)-f_n(v),\p_{uv}\rangle = 0\,.
\]

On the other hand, suppose that $q\in\cH$ satisfies $\langle q(u)-q(v),\p_{uv}\rangle=0$ for every edge $uv\in E$. Let $f\in D$. Then,
\begin{align*}
\langle q,Lf\rangle &= \sum_{u \in V}\left\langle q(u),\sum_{v\sim u}\p_{uv}\p_{uv}^*(f(u)-f(v)) \right\rangle \\
&= \sum_{uv\in E}\left\langle q(u)-q(v),\p_{uv}\p_{uv}^*(f(u)-f(v)) \right\rangle=0\,,
\end{align*}
where the last equality follows from our assumption on $q$. Note that this summation is finite since $f\in D$. Therefore, by the definition of the adjoint, $\bar L^*q=0$, and the proof follows by the assumption that $\bar L$ is self-adjoint.
\end{proof}

The following definition is natural, but will not be of much use in this work.
\begin{definition}
    A self-adjoint $d$-framework $(G,\p)$ is called $\ell_2$ \emph{strictly $d$-rigid} if \[\ker \bar L(G,\p)=\{0\}\,.\]
\end{definition}

No finite graph is strictly-rigid since \(\ker L\) contains the trivial motions.
For infinite graphs, $\ell_2$ strict rigidity is more subtle since the trivial motions need not belong to
\(\ell_2\). Consequently, \(\ker \bar L\) may be trivial even when
\(\p\) is generic, as the following example illustrates. The example
also shows that, on the other hand, there are generic
embeddings for which \(\ker L\) is non-trivial even for infinite graphs.

\begin{example}\label{ex:strictly_rigid}
Let $G$ be a bounded-degree graph on $V=\N$ such that the induced subgraph $G_n$ on
$[n]=\{1,\dots,n\}$ is $d$-rigid for every $n\in\N$. Let $\p:V\to[0,1]^d$ be a generic embedding which is \emph{well-spread}, in the sense that for some $\varepsilon>0$, and for every affine subspace $F\subset\R^d$, there exist infinitely many $v\in V$ such that
\[
\dist(\p(v),F)\ge \varepsilon.
\]
For example, draw $\p(v)$ uniformly and independently from $[0,1]^d$ for each $v\in V$.
By boundedness of $G$ and $\p$, the framework $(G,\p)$ is self-adjoint. Let $q\in\ker(\bar L)$. Restricting $q$ to the induced subgraph $G_n$, the
assumption that $G_n$ is $d$-rigid implies that $q|_{[n]}$ is induced by a trivial motion. Since this holds for every $n$, it follows from the genericity of $\p$ that $q$ is induced by a single global trivial motion. That is, there exist a skew-symmetric matrix $A$ and a vector $t$ such that
\[
q(v)=A\p(v)+t \qquad \text{for all } v\in V.
\]
The well-spreadness assumption on $\p$ implies that
\begin{equation}\label{eq:series}
\sum_{v\in V}\|A\p(v)+t\|^2 = \infty
\end{equation}
for all $(A,t)\neq(0,0)$. Since $q\in\cH=\ell_2(V;\C^d)$, this forces $A=0$ and $t=0$, and hence $q=0$. Therefore,
$
\ker(\bar L)=\{0\}.
$

Note that a different choice of a (possibly) generic embedding $\p$ may lead to non-zero elements of $\cH$ arising from trivial motions: indeed, if there exist a skew-symmetric matrix $A$ and a vector $t\in\R^d$, not both zero, such that \eqref{eq:series} converges, then the corresponding trivial motion belongs to $\cH$ and hence lies in $\ker(\bar L)$.
\end{example}

In the example above, every vertex is completely constrained: no vertex
can move in any direction. However, this phenomenon can be localized to
some directions but not others. For instance, if we add a new vertex \(v\)
and connect it to one vertex \(u\) of the original graph, then the motion
of \(v\) is constrained to lie in the subspace orthogonal to \(\p_{uv}\).
As before, such behavior cannot occur in finite frameworks, because
trivial motions are always present. The following lemma characterizes the
forbidden directions for a vertex.

\begin{lemma}\label{lem:forbidden_directions}
    Let $(G,\p)$ be a self-adjoint $d$-framework, and $v\in V(G)$ a vertex. Let $P:\cH\to\ker\bar L(G,\p)$ 
    be the orthogonal projection, and $P_{v,v}$ the $d\times d$ block corresponding to $v$. Then, for every $x\in \C^d$,
    \[
    x\in\ker P_{v,v} \iff \langle x,q(v)\rangle=0,~~\forall q\in\ker \bar L(G,\p)\,.
    \]
\end{lemma}
\begin{proof}
    Let $f\in D$ be defined by $f(v)=x$ and $f(u)=0$ for every $u\ne v$. Then,
    \[
    \|Pf\|^2 = \langle f,Pf\rangle = \langle x,(Pf)(v)\rangle = \langle x,P_{v,v}x\rangle\,.
    \]
    Therefore, since $P_{v,v,}$ is PSD, $P_{v,v}x=0$ if and only if $Pf=0$. Since $P$ is an orthogonal projection, this is equivalent to the condition that for every $q\in\ker \bar L(G,\p),$
    \[
    \langle x,q(v)\rangle = \langle f, q\rangle = 0\,,
    \]
    as claimed.
\end{proof}

\subsection{Random frameworks}\label{sec:admissible}
Due to the illustrated subtleties in generic rigidity of infinite frameworks, 
and because it is better suited to our methods, we shall often restrict the class of embeddings considered in this work as follows. 
Given a possibly random graph \(G\) and a probability measure \(\nu\) on
\(\R^d\), we write \((G,\nu)\) for the random \(d\)-framework \((G,\p)\)
in which the vectors \(\p(v)\), \(v\in V(G)\), are sampled independently
from \(\nu\).

We shall mostly consider probability measures \(\nu\) on \(\R^d\) that we
call \emph{admissible}: namely, \(\nu\) is supported on a bounded nonempty
open set \(U\subset\R^d\), and is absolutely continuous with respect to
Lebesgue measure on \(U\). For example, one may take \(\nu\) to be the
uniform distribution on \([0,1]^d\).

\section{Infinitesimal local flexibility}\label{sec:local_flex}
We now introduce the notion of local flexibility, which plays a central
role in this work. Although the notion can be defined for finite
frameworks, we formulate it in a more general setting.

Let $(G,\p)$ be a self-adjoint $d$-framework. 
Recall that 
    \[
    P=\mathbf 1_{\{0\}}(\bar L):\cH\to \ker(\bar L)
    \]
    and $P_{v,v}$ is the real $d\times d$ block of $P$ corresponding to a vertex $v$. 
\begin{definition}
Let $(G,\p)$ be a self-adjoint $d$-framework and $v\in V(G)$. The {\em local flexibility} of $v$ in $(G,\p)$ is defined by
    \[
    \phi_{G,\p}(v)=\tr(P_{v,v})\,.
    \]
\end{definition}
Observe that an alternative description of the local flexibility is given by
\[
\phi_{G,\p}(v) = \sum_{j=1}^d \mu_{e_{v,j}}(\{0\})\,,
\]
where $e_{v,j}$ is the unit vector of the $j$-th coordinate in the block corresponding to $v$, and $\mu_{e_{v,j}}$ is the spectral measure of $\bar L$ with respect to $e_{v,j}$.

When the context is clear, we sometimes omit the subscript $G,\p$ in $\phi$.

Intuitively, the local flexibility quantifies the degrees of freedom that $v$ has in the framework. 
In particular, it is easy to see that $d-\deg_G(v)\le \phi (v)\le d$. In a finite graph, $\phi(v)>0$ due to the trivial motions. 
But, Example \ref{ex:strictly_rigid} shows that the local flexibility can vanish in infinite graphs.

Crucially, by standard linear algebra, if $G$ is finite then the dimension of the kernel of the rigidity matrix $R$ is given by
\begin{equation}\label{eq:dimZ_local}
\dim Z(G,\p) = \sum_{v\in V(G)} \phi_{G,\p}(v)\,,
\end{equation}
since both sides are equal to the trace of $P$. Thus we  view the local flexibility of $v$ as the local ``contribution'' of $v$ to the dimension $Z$.

\subsection{Upper semi-continuity under Benjamini--Schramm convergence}
We now show that the local flexibility $\phi_{G,\p}(v)$ is a {\em local parameter}. That is, it is upper semi-continuous 
under Benjamini--Schramm convergence. This is a crucial property for our applications, as it allows us to deduce results about the local flexibility of finite graphs from results about the local flexibility of their infinite local limit.

This is a direct adaptation to the rigidity Laplacian of previous results on the weak convergence of spectral measures of graph operators under local weak convergence (e.g., \cite{BordenaveLelargeSalezDiluted,AbertThomViragSpectralMeasure}).

\begin{lemma}\label{lem:upper_semi}
    Suppose that $(G,o)$ is the local weak limit of a graph sequence $G_n$. Let $\nu$ be a probability measure on $\R^d$ for which $(G,\nu)$ and $(G_n,\nu)$ are all almost surely self-adjoint. Then,
    \[
        \limsup_{n\to\infty}\E[\phi_{G_n,\nu}(o_n)]\le \E[\phi_{G,\nu}(o)]\,,
    \]
    where $o_n\in V(G_n)$ is a uniform random vertex.
\end{lemma}
\begin{proof}
     By the Skorokhod representation theorem, we may assume that $(G_n,o_n)\to (G,o)$ almost surely in $\cG_\bullet$. By choosing appropriate embeddings of all graphs into a common vertex set $V$, we may further assume that $o_n=o\in V$ for all $n$, and that for every $r\ge 2$ there exists $n_r$ such that, almost surely,
    \begin{equation}\label{eq:equal_balls}
    B_{G_n}(o,r)=B_G(o,r)\qquad\text{for all } n\ge n_r.
    \end{equation}
    Fix $j\in[d]$ and let $e_j = e_{o,j}\in\ell_2(V;\C^d)$ be the unit vector corresponding to the $j$-th coordinate at the root $o$.

    Let $\p:V\to \R^d$ be such that $\p(v)$, $v\in V(G)$, are independent samples of $\nu$, and denote $\p_n=\p|_{V(G_n)}$.
    By assumption, the frameworks $(G_n,\p_n)$ and $(G,\p)$ are almost surely self-adjoint. Hence the associated rigidity Laplacians $\bar L_n = \bar L(G_n,\p_n)$ and $\bar L = \bar L(G,\p)$ admit spectral measures $\mu^{(n)}_{e_{j}}$ and $\mu_{e_{j}}$ with respect to $e_{j}$.

From \eqref{eq:equal_balls} and the assumption $\p_n=\p|_{V(G_n)}$, we deduce that for every \(f\in D\), \(L_n f=Lf\) for all sufficiently large \(n\),
since $L_n,L$ are equal in the blocks of the support of $f$. 
Due to the fact that \(D\) is a core for all \(\bar L_n\) and \(\bar L\), \cite[Theorem VIII.25(a)]{ReedSimonI} implies that \(\bar L_n\to \bar L\) in the strong resolvent sense. Therefore, we obtain almost sure weak convergence of the spectral measures $\mu^{(n)}_{e_{j}}\Rightarrow \mu_{e_{j}}$. By the Portmanteau theorem,
    \[
        \limsup_{n\to\infty}\mu^{(n)}_{e_{j}}(\{0\})\le \mu_{e_{j}}(\{0\})\,.
    \]
    Summing over $j=1,\dots,d$ and recalling that $\phi_{G,\p}(o)=\sum_{j=1}^d \mu_{e_{o,j}}(\{0\})$, we obtain, almost surely,
    \[
        \limsup_{n\to\infty}\phi_{G_n,\p_n}(o_n)\le \phi_{G,\p}(o)\,.
    \]
    The proof then follows by Fatou's lemma, using that $0\le \phi_{G_n,\p_n}(o_n)\le d$.
\end{proof}

\section{Local flexibility in trees}\label{sec:trees}
\subsection{Self-adjointness of trees}
From a classical rigidity-theory perspective, finite trees are fairly easy to analyze and are arguably not particularly interesting. However, as we see here, from the perspective of local flexibility, trees---and infinite trees in particular---are nontrivial and very interesting to study. 

We start this investigation by formulating a useful sufficient condition for self-adjointness of frameworks whose underlying graph 
is an infinite tree $T$. We say that a finite subset $S\subset V(T)$ that induces a connected subtree is $C$-bounded, for some $C>0$, if no vertex in $S$ has more than $C$ neighbors outside of $S$.
\begin{lemma}\label{lem:tree_self_adjoint}
    Let $(T,o)$ be a locally finite rooted tree and $\p:V(T)\to\R^d$ such that $\{\p_{uv}\mid uv\in E\}$ is bounded. If there exists $C>0$ such that $V(T)$ is covered by $C$-bounded subsets containing $o$, then $(T,\p)$ is self-adjoint.
\end{lemma} 
\begin{proof}
    By \cite[Theorem VIII.3]{ReedSimonI}, to show that $\bar L=\bar L(T,\p)$ is self-adjoint it suffices to prove that there is no non-zero vector $\psi$ in the domain of $\bar L^*$ for which $L^*\psi=\pm \ii\psi$. We consider the case $\bar L^*\psi=-\ii\psi$, and note that the proof of the other case is nearly identical. For a vertex $u$, let $\psi_u\in D$ be defined by $\psi_u(u)=\psi(u)$ and $\psi_u(v)=0$ for every $v\ne u$. Therefore,
    \[
        \bar L\psi_u(v) = L\psi_u(v) =
        \begin{cases}
        \displaystyle \sum_{w\sim u} \langle \psi(u),\p_{uw}\rangle\p_{uw}, & u=v,\\[1ex]
        \displaystyle -\,\langle \psi(u),\p_{uv}\rangle\p_{uv}, & \{u,v\}\in E,\\[1ex]
        0, & \text{otherwise}\,.
        \end{cases}
    \]
    Hence,
    \[
        \langle \bar L\psi_u,\psi\rangle =\sum_{v\sim u}\langle\psi(u),\p_{uv}\rangle\langle\p_{uv},\psi(u)-\psi(v)\rangle\,. 
    \]
    On the other hand, by our assumption on $\psi$,
    \[
        \langle \bar L\psi_u,\psi\rangle = \langle \psi_u,\bar L^*\psi \rangle = +\ii\|\psi(u)\|^2\,.
    \]    
    Therefore, for any finite subset $S\subset V(T)$ there holds
    \[
        \ii\sum_{u\in S}\|\psi(u)\|^2 =
        \sum_{uv\in E(S)}|\langle\psi(u)-\psi(v),\p_{uv}\rangle|^2
        +\sum_{uv\in E(S,S^c)}\langle\psi(u),\p_{uv}\rangle\langle\p_{uv},\psi(u)-\psi(v)\rangle\,. 
        \]
        Here $E(S)$ denotes the edges induced by subset $S$, and $uv\in E(S,S^c)$ denotes that $u\in S,v\notin S$ and $u\sim v$. By considering the imaginary part of this equality and applying the triangle and Cauchy--Schwarz (CS) inequalities we obtain 
        \begin{equation}\label{eq:sa_norm}
        \sum_{u\in S}\|\psi(u)\|^2 \leq
        \sum_{uv\in E(S,S^c)} \|\psi(u)\|\cdot \|\p_{uv}\|^2\cdot \|\psi(v)\|\,.
        \end{equation}
        Suppose that $S$ is $C$-bounded, and that $\|\p_{uv}\|^2\le C$ for all $uv\in E$. In such a case, we find by applying CS twice that
        \begin{align*}
            \sum_{u\in S}\|\psi(u)\|^2 
            \leq ~&C\cdot \sum_{u\in S} \|\psi(u)\|\cdot C^{1/2}\cdot\left(\sum_{v\sim u,v\notin S}\|\psi(v)\|^2\right)^{1/2}\\
            \leq~
            &C^{1.5}\cdot \left(\sum_{u\in S} \|\psi(u)\|^2\right)^{1/2}\cdot\left(\sum_{u\in S}\sum_{v\sim u,v\notin S}\|\psi(v)\|^2\right)^{1/2} \\
            \leq~
            & C^{1.5}\cdot \left(\sum_{u\in S} \|\psi(u)\|^2\right)^{1/2}\cdot\left(\sum_{v\notin S}\|\psi(v)\|^2\right)^{1/2}\,.
        \end{align*}
    The last inequality follows from the assumption that $S$ induces a connected subtree, hence every $v\notin S$ has at most one neighbor in $S$. A straightforward manipulation yields
        \[
        \sum_{u\in S}\|\psi(u)\|^2 \leq \frac{C^3}{C^3+1}\|\psi\|^2\,.   
        \]
    On the other hand, the family of $C$-bounded subsets that contain the root is closed under taking finite unions, whence our assumption that $V(T)$ is covered by such sets implies that the supremum of $\sum_{u\in S}\|\psi(u)\|^2$ over all $C$-bounded subsets $S$ is $\|\psi\|^2$, yielding a contradiction.
\end{proof}

\subsection{Forward flexibility matrices and their recursive formula}\label{sec:recursive}

For a rooted tree $(T,o)$ and a vertex $v\in V(T)$ we denote by $T^v$ the subtree of $T$ induced by $v$ and its descendants --- all the vertices $u$ for which the shortest path in $T$ from $o$ to $u$ passes through $v$ --- 
and view it as rooted at $v$. 
In addition, for an embedding $\p:V\to \R^d$, we denote by $\p^v=\p|_{T^v}$. 

It is easy to see that if $(T,\p)$ is self-adjoint then so is $(T^v,\p^v)$. In such a case, we consider the orthogonal projection
\[
P^v=\mathbf 1_{\{0\}}(\bar L(T^v,\p^v)):\ell_2(V(T^v);\C^d)\to\ker(\bar L(T^v,\p^v))\,,
\]
and call its restriction $Q_v=(P^v)_{vv}$ to the $d\times d$ block corresponding to the vertex $v$ the {\em forward flexibility matrix} of $v$ in $T$.
Note that $Q_v$ has real entries, and recall that its trace is the local flexibility of $v$ in $(T^v,\p^v)$. 
In addition, by Lemma \ref{lem:forbidden_directions}, the kernel of $Q_v$ consists of the forbidden directions for $v$ in this framework $(T^v,\p^v)$.

The next theorem finds a recursive formula for computing $Q_o$ using the vectors $\p_{ov}$ and the matrices $Q_v$, over all children $v$ of the root $o$.

\begin{theorem}\label{thm:recursion}
Let $(T,\p)$ be a self-adjoint $d$-framework whose underlying graph $T$ is a tree rooted at $o$. Denote by
$v_1,\dots,v_m$ the children of the root $o$. Let
\[
H=\{\,i\in[m]\mid Q_{v_i}\p_{ov_i}=0\,\},
\]
$W\le \C^d$ be the orthogonal complement of
$\mathrm{span}_{\C}\{\p_{ov_i}\mid i\in H\}$, and $P^W:\C^d\to W$ be the orthogonal projection. Then,
\begin{enumerate}
\item $Q_o\p_{ov_i}=0$ for every $i\in H$.
\item
\[
Q_o|_W
=
\left(
\Id_W+\sum_{i\notin H}\frac{P^W\p_{ov_i}(P^W\p_{ov_i})^*}{\p_{ov_i}^*Q_{v_i}\p_{ov_i}}
\right)^{-1}.
\]
\end{enumerate}
\end{theorem}

\begin{proof}
Write $\bar L=\bar L(T,\p)$ and $\bar L^{v_i}=\bar L(T^{v_i},\p^{v_i})$ for $i\in[m]$.
For the first item, let $i\in H$. For every $q\in\ker\bar L$ there holds that $q|_{T^{v_i}}\in\ker \bar L^{v_i}$ by
the characterization of the kernel in Lemma~\ref{lem:kernel}. Hence, by Lemma~\ref{lem:forbidden_directions},
\[
\langle q(v_i),\p_{ov_i}\rangle=\langle q|_{T^{v_i}}(v_i),\p_{ov_i}\rangle=0\,,
\]
due to the assumption that $i\in H$. On the other hand, by Lemma~\ref{lem:kernel} we have
\[
\langle q(o)-q(v_i),\p_{ov_i}\rangle=0\,.
\]
Therefore, $\langle q(o),\p_{ov_i}\rangle=0$. The claim follows from Lemma~\ref{lem:forbidden_directions}.

We proceed to the second item. Let $P$ be the orthogonal projection onto $\ker\bar L$.
For $z=\ii\cdot s$ with $s>0$ real, set
\[
R(z):=(\bar L-zI)^{-1}.
\]
Define $h_z(\lambda):=-z(\lambda-z)^{-1}$ for $\lambda\in\R$. By the functional calculus,
$h_z(\bar L)=-zR(z)$ is a bounded operator. For $\lambda\neq 0$ we have $h_z(\lambda)\to 0$ as $z\to 0$, while $h_z(0)=1$.
Moreover,
\[
|h_z(\lambda)|=\frac{|z|}{|\lambda-z|}\le 1\,,
\]
since $z$ is purely imaginary. Thus $h_z\to \mathbf 1_{\{0\}}$ pointwise and $(h_z)$ is uniformly bounded, so by dominated
convergence in the spectral theorem,
\[
h_z(\bar L)f\to Pf\qquad\text{for every }f\in\cH
\]
as $z\to 0$. In particular, applying this entrywise to the $d\times d$ block at $o$ gives
\[
-zR(z)_{o,o}\to P_{o,o}=Q_o\quad \mbox{as }z\to 0.
\]

Similarly, let
\[
R'(z):=\left(\sum_{i=1}^m \bar L^{v_i}-zI\right)^{-1}.
\]
Since $T$ is a tree, $\sum_{i=1}^m\bar L^{v_i}$ is a direct sum, hence so is $R'(z)$, and therefore for each $i\in[m]$,
\[
-zR'(z)_{v_iv_i}\to Q_{v_i}
\quad \mbox{as }z\to 0.
\]

We now derive the recursion. We can write
\[
\bar L=\sum_{i=1}^m \bar L^{v_i}+M,
\]
where $M$ is the rigidity Laplacian of the star at $o$. 
More explicitly, by letting $N_i=\p_{ov_i}\p_{ov_i}^*$ for $i\in [m]$, we have
\[
M_{uv}=
\begin{cases}
\sum_{j=1}^m N_j, & (u,v)=(o,o),\\
-N_i, & (u,v)=(o,v_i)\ \text{or}\ (v_i,o),\\
N_i, & (u,v)=(v_i,v_i),\\
0, & \text{otherwise}.
\end{cases}
\]
By the tree structure,
\begin{equation}\label{eq:Rprime_blocks}
R'_{o,o}(z)=-z^{-1}I,\qquad R'_{o,v_i}(z)=0,\qquad R'_{v_i,v_j}(z)=0\ \ (i\neq j).
\end{equation}

The second resolvent identity gives $R(z)-R'(z)=-R(z)MR'(z)$ (we sometimes omit the dependence on $z$ for brevity). 
Recalling \eqref{eq:Rprime_blocks}, we obtain that for every $i\in[m]$,
\begin{align*}
R_{o,v_i}
&=-R_{o,o}M_{o,v_i}R'_{v_i,v_i}-R_{o,v_i}M_{v_i,v_i}R'_{v_i,v_i}\\
&=R_{o,o}N_iR'_{v_i,v_i}-R_{o,v_i}N_iR'_{v_i,v_i}.
\end{align*}

so
\begin{equation}\label{eq:Rovi}
R_{o,v_i}=R_{o,o}N_iR'_{v_i,v_i}\,(I+N_iR'_{v_i,v_i})^{-1}.
\end{equation}
Note that in \eqref{eq:Rovi} we used that $I+N_iR'_{v_i,v_i}$ is invertible. This holds since $N_i=\p_{ov_i}\p_{ov_i}^*$, whence matrix determinant lemma yields
\[
\det(I+N_iR'_{v_i,v_i})=1+\p_{ov_i}^*R'_{v_i,v_i}\p_{ov_i}.
\]
For $z=\ii s$ with $s>0$, since $R'$ is the resolvent of a self-adjoint operator, we have
\[
\Im\bigl(\p_{ov_i}^*R'_{v_i,v_i}\p_{ov_i}\bigr)>0.
\]
In particular, $\p_{ov_i}^*R'_{v_i,v_i}\p_{ov_i}\notin\R$,
hence the determinant above is nonzero.  

Next, using the Sherman--Morrison formula for rank-one updates, \eqref{eq:Rovi} simplifies to
\begin{equation}\label{eq:Rovi_simplified}
R_{o,v_i}
=
R_{o,o}\cdot \frac{N_iR'_{v_i,v_i}}{1+\p_{ov_i}^*R'_{v_i,v_i}\p_{ov_i}}\,.
\end{equation}

Applying again $R-R'=-RMR'$ for the $(o,o)$ block, and using \eqref{eq:Rprime_blocks}, gives
\begin{align*}
R_{o,o}+z^{-1}I
&=-R_{o,o}M_{o,o}R'_{o,o}-\sum_{i=1}^m R_{o,v_i}M_{v_i,o}R'_{o,o}\\
&=z^{-1}\left(R_{o,o}\sum_{i=1}^m N_i-\sum_{i=1}^m R_{o,v_i}N_i\right)\\
&=z^{-1}R_{o,o}\sum_{i=1}^m\left(N_i-\frac{N_iR'_{v_i,v_i}N_i}{1+\p_{ov_i}^*R'_{v_i,v_i}\p_{ov_i}}\right)\\
&=z^{-1}R_{o,o}\sum_{i=1}^m\frac{N_i}{1+\p_{ov_i}^*R'_{v_i,v_i}\p_{ov_i}}\,,
\end{align*}
where in the third equality we used \eqref{eq:Rovi_simplified}, and in the fourth equality we used
$N_iR'_{v_i,v_i}N_i=(\p_{ov_i}^*R'_{v_i,v_i}\p_{ov_i})\,N_i$.
Rearranging yields the recursion
\begin{equation}\label{eq:recursion}
-zR_{o,o}(z)\left(I+\sum_{i=1}^m\frac{\p_{ov_i}\p_{ov_i}^*}{-z+\p_{ov_i}^*(-zR'_{v_i,v_i}(z))\p_{ov_i}}\right)=I.
\end{equation}

Now fix $w\in W$. Since $\p_{ov_i}^*w=0$ for $i\in H$, multiplying \eqref{eq:recursion} by $w$ gives
\[
-zR_{o,o}(z)\left(
w+\sum_{i\notin H}\frac{\p_{ov_i}\p_{ov_i}^*w}{-z+\p_{ov_i}^*(-zR'_{v_i,v_i}(z))\p_{ov_i}}
\right)=w.
\]
Letting $z\to 0$ along the imaginary axis, we use $-zR_{o,o}(z)\to Q_o$ and $-zR'_{v_i,v_i}(z)\to Q_{v_i}$.
Moreover, for $i\notin H$ we have $\p_{ov_i}^*Q_{v_i}\p_{ov_i}\neq 0$ by definition of $H$, hence the denominators converge to
$\p_{ov_i}^*Q_{v_i}\p_{ov_i}$. Therefore,
\[
Q_o\left(
w+\sum_{i\notin H}\frac{\p_{ov_i}\p_{ov_i}^*w}{\p_{ov_i}^*Q_{v_i}\p_{ov_i}}
\right)=w\,.
\]
Finally, by the first item we have $Q_oP^W=Q_o$, and clearly $P^Ww=w$ for $w\in W$. Hence, replacing $\p_{ov_i}$ by $P^W\p_{ov_i}$ inside the
rank-one terms does not change the action on $W$, and we obtain
\[
Q_o|_W=
\left(
\Id_W+\sum_{i\notin H}\frac{P^W\p_{ov_i}(P^W\p_{ov_i})^*}{\p_{ov_i}^*Q_{v_i}\p_{ov_i}}
\right)^{-1},
\]
as claimed.
\end{proof}

In the next lemma we study the event that in the random framework $(T,\nu)$, the forward flexibility matrix of the root $o$ of $T$ is trivial.

\begin{lemma}\label{lem:Qo0}
    Let $\nu$ be an admissible probability measure on $\R^d$, and let $(T,o)$ be a rooted tree such that $(T,\nu)$ is almost surely self-adjoint. 
    Then, almost surely, $Q_o=0$ if and only if $o$ has at least $d$ children $v$ for which $Q_v=0$. 
    In particular, the event $Q_o=0$ is independent of $\p(o)$.
\end{lemma}
Note that for a fixed tree \((T,o)\), the lemma shows that \(Q_o=0\) is a tail event, and therefore admits a \(0/1\) law.
However, we do not know whether the probability of this event can depend on the admissible distribution \(\nu\), or whether it depends only on the tree structure.

\begin{proof}
First, we claim that almost surely, for every edge $uv\in E(T)$ with $u$ the parent of $v$, we have
\begin{equation}\label{eq:QvpuvQv0}
Q_v\p_{uv}=0 \quad\Longleftrightarrow\quad Q_v=0.    
\end{equation}
Indeed, $\p(u)$ is independent of $Q_v$, and by our assumption on $\nu$, the vector $\p_{uv}$ almost surely does not belong to $\ker Q_v$, unless $Q_v=0$.

Denote by $v_1,\ldots,v_m$ the children of the root $o$. By \eqref{eq:QvpuvQv0}, the set $H$ from Theorem \ref{thm:recursion} is almost surely equal to
\begin{equation}\label{eq:H}
H=\{i\in[m]\mid Q_{v_i}=0\}.
\end{equation}
Theorem~\ref{thm:recursion} implies that $Q_o=0$ if and only if the subspace $W$, which is the orthogonal complement of the linear span of $\{\p_{ov_i}: i\in H\}$, is trivial.
By our assumption on $\nu$, this occurs if and only if $|H|\ge d$, as claimed.
\end{proof}

\subsection{Galton--Watson frameworks}
\label{sec:GW}
We now turn to study the local flexibility of a Galton--Watson $\mathrm{GW}(X,Y)$ branching process $(T,o)$ as defined in Section \ref{sec:preliminaries}.
Here $X,Y$ are probability distributions on $\N_{\ge 0}$ that describe the offspring distribution of the root and every other vertex respectively.
Mostly, we will be interested in the case where $Y$ is the size-biasing of $X$. Also, we restrict to frameworks of the form $(T,\nu)$, where $\nu$ is an admissible probability measure on $\R^d$. 

First, we state a sufficient condition for the self-adjointness of such frameworks.
\begin{claim}\label{clm:GWSA}
    Let $T\sim\mathrm{GW}(X,Y)$ and $\nu$ be an admissible probability measure on $\R^d$. 
    If $\E[Y]<\infty$ then $(T,\nu)$ is almost surely self-adjoint.
\end{claim}
Note that if $Y$ is the size-biasing of $X$ then $\E[Y]<\infty$ if and only if $\E[X^2]<\infty$.
\begin{proof}
There exists $C\in\N$ such that
\begin{equation}\label{eq:chooseC}
\sum_{m>C} m\,Y(m)<1.
\end{equation}
We claim that, almost surely, for every vertex $v$ there exists a finite subtree $F^v$ of $T$, rooted at $v$, such that every leaf of $F^v$
has at most $C$ children in $T$. Clearly, it suffices to prove this for every non-root vertex $v$.

Reveal the tree $T$ from $v$ outward, but whenever an explored vertex has at most $C$ children, declare it a leaf and do not explore its children.
Let $F^v$ be the subtree revealed by this process. Then $F^v$ is a Galton--Watson tree rooted at $v$ with offspring distribution $Y'$ given by
\[
Y'(m)=
\begin{cases}
\P(Y\le C), & m=0,\\
0, & 0<m\le C,\\
Y(m), & m>C,
\end{cases}
\]
with expected number of offspring equal to
\[
\E[Y']=\sum_{m>C} m\,Y(m)<1\,.
\]
Hence, $F^v$ is almost surely finite. Moreover, by construction, every leaf of $F^v$ has at most $C$ children in $T$.

Let $S_r\subseteq V(T)$ consist of $o$ and the union of $F^v$ over all vertices $v$ at distance at most $r$ from $o$.
Then $S_r$ induces a connected finite subtree, contains the root $o$, and it is $C$-bounded. In addition, $\bigcup_{r\ge 0} S_r=V(T)$.
Since $\nu$ has a bounded support, the claim follows from Lemma~\ref{lem:tree_self_adjoint}.
\end{proof}

The main result of this section solves the recursion in Theorem~\ref{thm:recursion} for Galton--Watson frameworks in which $Y$ is the size-biasing of $X$,
and the embeddings are sampled independently from an admissible distribution $\nu$.

\begin{theorem}\label{thm:Exp_loc_flex}
    Suppose that $X$ is a probability distribution on $\N_{\ge 0}$ with a bounded second moment, $Y$ the size-biasing of $X$, $T\sim \mathrm{GW}(X,Y)$ with root $o$, and let $\nu$ be a probability measure on $\R^d$ as above. Then, there exists a solution $p \in [0,1]$ to $p=\P(Y_p\ge d)$ such that
    \begin{equation}\label{eq:Xformula}
    \E[\phi_{T,\nu}(o)] = \E[\mathbf 1_{X_p<d}\cdot(d-X_{p})] - \frac12(1-p)^2\E[X]\,.
    \end{equation}
    
\end{theorem}
\begin{proof}
Let $(T',u)$ denote a $\mathrm{GW}(Y,Y)$-distributed tree rooted at $u$. Then, $(T',\nu)$ is almost surely self-adjoint by Claim \ref{clm:GWSA}. 
Let $p_Y=\P(Q_u=0)$, where $Q_u$ denotes the forward flexibility matrix of $u$ in $T'$. 
We claim that the number of children $v$ of $u$ in $T'$ for which $Q_v=0$ is $Y_{p_Y}$-distributed. 
Indeed, the total number of children $v$ of $u$ is $Y$-distributed, and the events $Q_v=0$ are independent, each with probability $p_Y$, since the subtree rooted at $v$ is also $\mathrm{GW}(Y,Y)$-distributed.
By Lemma \ref{lem:Qo0} we find that 
\begin{equation}\label{eq:px-py-bin-recursion}
p_Y=\P\!\left(Y_{p_Y}\ge d\right)\,.
\end{equation}

If \(p_Y=1\), then \eqref{eq:px-py-bin-recursion} gives
\(\mathbb P(Y\ge d)=1\), and hence
    $\mathbb P(1\le X\le d)=0$.
Substituting \(p=1\) into the right-hand side of
\eqref{eq:Xformula} therefore gives \(dX(0)\).
On the other hand, the root \(o\) of \(T\) is isolated with probability
\(X(0)\). In this case its local flexibility is \(d\). Otherwise, \(o\)
has at least \(d+1\) children, and all their forward flexibility matrices
vanish. Hence, the local flexibility at the root is \(0\) by
Lemma~\ref{lem:Qo0}. Thus the expected local flexibility is \(dX(0)\),
so \eqref{eq:Xformula} holds also in this case.

So we assume $p_Y<1$. The following distribution plays a key role in the proof: sample a rooted $\mathrm{GW}(Y,Y)$-tree $(T',u)$ and an embedding $\p:V(T')\to\R^d$ by drawing i.i.d.\ samples from $\nu$. Let $\cD^Y$ denote the law of the pair $(\p(u),Q_u)$ under this procedure, conditioned on the event $Q_u\neq 0$.
By Lemma~\ref{lem:Qo0}, the marginal law of $\p(u)$ under $\cD^Y$ is $\nu$, since $\p(u)$ is independent of the event $Q_u=0$. 
However, the matrix $Q_u$ itself may depend on $\p(u)$.

\medskip

We now turn our attention to $(T,o)$. Note that $(T,\nu)$ is almost surely self-adjoint by Claim \ref{clm:GWSA}. 
By \eqref{eq:QvpuvQv0}, the subset $H$ from Theorem \ref{thm:recursion} is almost surely equal to the subset of the children $v$ of the root $o$ for which $Q_v=0$. We claim that $|H|$ is $X_{p_Y}$-distributed. Indeed, the total number of children $v$ of $o$ is $X$-distributed, and, as in $T'$, the events $Q_v=0$ are independent of probability $p_Y$. 

Let $(\p(u_i),Q_{u_i})_{i\ge 1}$ be a sequence of i.i.d.\ samples from $\cD^Y$, and $\p(o),\p(v_i),~i=1,\ldots,d-1$ independently drawn from $\nu$. Via Theorem \ref{thm:recursion}, we can sample a matrix $Q$ with the same distribution as the forward flexibility matrix $Q_o$ of $o$ in $(T,\nu)$
 using these random vectors and matrices: 
Indeed, start by sampling \((h,s)\sim (X_{p_Y},X_{1-p_Y})\).
Viewing \(h\) as the size of \(|H|\), we set \(Q=0\) if \(h\ge d\).
Otherwise, view
\begin{enumerate}
    \item[(i)] \(v_1,\ldots,v_h\) as the children of the root \(o\)
    for which \(Q_{v_i}=0\) (note that \(\p(v_i)\) is
    \(\nu\)-distributed even under the conditioning \(Q_{v_i}=0\)),
    \item[(ii)] \(u_1,\ldots,u_s\) as the remaining children of the root,
    and
    \item[(iii)] \(Q_{u_i}\) as the forward flexibility matrix of \(u_i\)
    in \(T\) (matching the definition of \(\cD^Y\)).
\end{enumerate}

In this realization, $W$ is the orthogonal complement of
\[
\mathrm{span}\{\p_{ov_i}: 1\le i\le h\},
\]
and has dimension $d-h$. Set 
\[
A_i=\frac{P^W\p_{ou_i}(P^W\p_{ou_i})^*}{\p_{ou_i}^*Q_{u_i}\p_{ou_i}}\,,\qquad i\ge 1\,,
\]
which is a $(d-h)\times(d-h)$ matrix acting on $W$. Theorem~\ref{thm:recursion} implies that the random $d\times d$ matrix $Q$ defined by $Q|_{W^\perp}=0$ and
\[
Q|_W=\Bigl(\Id_W+\sum_{i=1}^{s}A_i\Bigr)^{-1}
\]
has the distribution of the forward flexibility matrix $Q_o$ of the root $o$ in $T$.
Therefore,

\begin{align}
\notag 
\E[\tr\,Q_o] 
&=~ \E\left[ \mathbf 1_{h<d}\cdot \tr\left(\left(\Id_W + \sum_{i=1}^{s}A_i\right)^{-1}\right)\right] \\
\notag &=~ \E\left[\mathbf 1_{h<d}\cdot \tr\left(
\Id_W-\sum_{i=1}^{s}A_i \left(\Id_W + \sum_{i'\in[s]\setminus\{i\}}A_{i'} +A_i\right)^{-1}\right)\right] \\
\label{eq:first_stop} &=~ \E[\mathbf 1_{h<d}\cdot (d-h)] - \E_{s\sim X_{1-p_Y}}\left[ s\cdot f(s) \right]\,,
\end{align}
where $f:\N\to \R$ is defined by
\[
f(s) = 
\E\left[\mathbf{1}_{h<d}\cdot 
\tr\left(A_s \left(\Id_W + \sum_{i=1}^{s-1}A_i+A_s\right)^{-1}\right)
~~\middle |~~s\right]\,.
\]
In the align, the transition to the second line is by the equality $(I+S)^{-1}=I-S(I+S)^{-1}$ which holds if $I+S$ is invertible, and the next transition is by linearity, $\dim W=d-h$, and the fact that the $A_i$'s have the same distribution. 

By \eqref{eq:sizeBiasF}, the equality $\E[X_{1-p_Y}]=(1-p_Y)\E[X]$, and the first item in Claim \ref{clm:sizeBias_new_claim}, we reformulate \eqref{eq:first_stop} to
\begin{equation}\label{eq:ft}
\E[\tr\,Q_o]=\E[\mathbf 1_{X_{p_Y}<d}\cdot (d-X_{p_Y})]- 
(1-p_Y)\E[X]\cdot\E_{t\sim Y_{1-p_Y}}\left[f(t+1) \right]\,.    
\end{equation}

We turn to study the expectation of $f(t+1)$, where $t\sim Y_{1-p_Y}$. 
By definition, when evaluating $f(t+1)$ we draw
$h\sim X_{p_Y}\mid \{X_{1-p_Y}=t+1\}$. Therefore, by the third item of
Claim~\ref{clm:sizeBias_new_claim}, the joint distribution of $(h,t)$ is
$(Y_{p_Y},Y_{1-p_Y})$. Once $h$ is sampled, both $W$ and the matrices $A_i$ are defined via the randomness $(\p(u_i),Q_{u_i})_{i\ge 1}$ together with the independent samples $\p(o)$ and $\p(v_i)_{i<d}$. We denote by $\E'$ the expectation taken over $(h,t)\sim(Y_{p_Y},Y_{1-p_Y})$ and all this auxiliary randomness, and obtain
\begin{equation}\label{eq:Exp_of_ft}
    \E_{t\sim Y_{1-p_Y}}\!\left[f(t+1) \right]
    = \E'\!\left[ 
        \mathbf{1}_{h<d}\cdot 
\tr\!\left(A_{t+1} \left(\Id_W + \sum_{i=1}^{t}A_i+A_{t+1}\right)^{-1}\right)
    \right].
\end{equation}

Condition on the value of $h<d$. In such a case, we let $B_t$ be the matrix of order $d-h$ acting on $W$ defined by 
$$B_t=\Id_W + \sum_{j=1}^{t}A_j\,.$$ Also, denote $u=u_{t+1}$, $Q_u=Q_{u_{t+1}}$ for brevity. By the Sherman--Morrison formula and the definition of $A_{t+1}$, we find that
\[
A_{t+1}(B_t+A_{t+1})^{-1} = \frac{A_{t+1}B_t^{-1}}{1+\frac{(P^W\p_{ou})^*B_t^{-1}(P^W\p_{ou})}{\p_{ou}^*Q_{u}\p_{ou}}}\,.
\]
Therefore, the cyclic property of the trace and a simple manipulation give
\begin{align}
\tr\!\left[A_{t+1}(B_t+A_{t+1})^{-1}\right]
&= \frac{(P^W\p_{ou})^*B_t^{-1}(P^W\p_{ou})}{\p_{ou}^*Q_{u}\p_{ou}+(P^W\p_{ou})^*B_t^{-1}(P^W\p_{ou})} \nonumber\\[0.6em]
&= \frac{\p_{ou}^*\tilde Q\,\p_{ou}}{\p_{ou}^*Q_{u}\p_{ou}+\p_{ou}^*\tilde Q\,\p_{ou}}\,, \label{eq:trace_A_over_BplusA}
\end{align}
where $\tilde Q$ denotes a $d\times d$ matrix defined by $\tilde Q|_{W^\perp}=0$ and $\tilde Q|_W = B_t^{-1}$.

By combining \eqref{eq:Exp_of_ft} and \eqref{eq:trace_A_over_BplusA} we find
\begin{equation}\label{eq:almost_there}
\E_{t\sim Y_{1-p_Y}}[f(t+1)]=\E'(\mathbf 1_{h<d})\E'\left[ \frac{\p_{ou}^*\tilde Q\,\p_{ou}}{\p_{ou}^*Q_{u}\p_{ou}+\p_{ou}^*\tilde Q\,\p_{ou}} ~~\middle |~~ h<d \right]\,.  
\end{equation}

Crucially, observe that, by combining Theorem \ref{thm:recursion} and the definition of $B_t$, we can view $\tilde Q$ as a forward flexibility matrix. More precisely, since $(h,t)\sim (Y_{p_Y},Y_{1-p_Y})$, we derive that, conditioned on $h<d$, the pair $(\p(o),\tilde Q)$ is $\cD^Y$-distributed.
In addition, it is independent of the $\cD^Y$-distributed pair $(\p(u),Q_u)$ (which by definition denotes the pair $(\p(u_{t+1}),Q_{u_{t+1}})$).
Therefore, conditioned on $h<d$, the pair of positive random variables
$\p_{ou}^*Q_{u}\p_{ou}$ and 
$\p_{ou}^*\tilde Q\,\p_{ou}$
are exchangeable: indeed, the pairs \((\p(u),Q_u)\) and \((\p(o),\tilde Q)\) are independent and $\cD^Y$-distributed, and exchanging them only replaces \(\p_{ou}\) by \(-\p_{ou}\). Hence,
\[
\E'\left[ \frac{\p_{ou}^*\tilde Q\,\p_{ou}}{\p_{ou}^*Q_{u}\p_{ou}+\p_{ou}^*\tilde Q\,\p_{ou}} ~~\middle |~~ h<d \right]=\frac{1}{2}\,.
\]
Combined with $\E'[\mathbf 1_{h<d}]=1-p_Y$ which follows from \eqref{eq:px-py-bin-recursion}, we find that \eqref{eq:almost_there} reformulates to
\[
\E_{t\sim Y_{1-p_Y}}[f(t+1)]=\frac{1-p_Y}{2}\,.
\]
Substituting this equality into \eqref{eq:ft} yields
\begin{align*}
    \E[\tr\,Q_o]
    &=~
\E[\mathbf 1_{X_{p_Y}<d}\cdot(d-X_{p_Y})] - \frac12(1-p_Y)^2\E[X]\,,
\end{align*}
and together with~\eqref{eq:px-py-bin-recursion}, the theorem follows.

\end{proof}

\section{Applications}\label{sec:applications}
We now apply Theorem \ref{thm:Exp_loc_flex} to several families of both random and non-random graphs. 
\subsection{Random graphs with a given degree distribution}\label{sec:X}
In this subsection we describe a fairly broad family of graph sequences for which the trivial upper bound \eqref{eq:core-rank-upper-bound} for the \(d\)-rigidity rank is attained asymptotically. 
This general treatment may have further applications, and may also shed light on the connection between the rigidity rank and the \((d+1)\)-core in random graphs. We stress that some of the assumptions below are likely not optimal. However, since weakening them is not needed for the main theorems of the paper, we have not pursued this further.

Let $X$ be a probability distribution on $\N_{\ge 0}$ and $Y$ its size-biasing. 
Assume that $X$ has finite first and second moments. Denote by
\[
\phi(p)=\E\big[\mathbf 1_{X_p<d}\,(d-X_p)\big]-\frac12(1-p)^2\E[X]
\]
the right-hand side of \eqref{eq:Xformula}. In addition, let $\hat p=\hat p(X)$ be the largest solution in $[0,1]$ for \eqref{eq:px-py-bin-recursion}. We note that in Section~\ref{sec:intro}, the notation \(\hat p\) was used
for this same quantity in the special case \(X=\operatorname{Poi}(c)\). The following condition on \(X\) will be used below.

\begin{condition}\label{ass:max_phi}
    $\phi(p)$ attains its maximum in $[0,1]$ at either $p=0$ or $p=\hat p$.
\end{condition}

In the next subsections we show that the distributions $X$ relevant to our main theorems ($X=\operatorname{Poi}(c)$ and $X\equiv k$) satisfy Condition \ref{ass:max_phi} using the following claim.

\begin{claim}\label{clm:max_phi}
    If $\phi''(p)$ is unimodal, i.e., there exists $p_*\in[0,1]$ such that $\phi''$ is increasing on $(0,p_*]$
and decreasing on $[p_*,1)$, then Condition \ref{ass:max_phi} holds.
\end{claim}
\begin{proof}
Straightforward calculations show that
\[
\phi'(p)\;=\;\E[X]\cdot\Big(\P\big(Y_p\ge d\big)-p\Big)\,,
\]
and
\begin{equation}\label{eq:phi''}
\phi''(p)
=
\E[X]\cdot\left(p^{-1}d\cdot \P(Y_p=d)-1\right)\,,
\end{equation}
for $0<p\le 1$.

     Note that since $d\ge 2$, we have $\phi''(0)=-\E[X]<0$, $\phi'(0)=0$ and $\phi'(1)=\E[X]\big(\P(Y\ge d)-1\big)\le 0$. Since $\phi''$ is first non-decreasing and then non-increasing, and since
$\phi''(0)<0$, the derivative $\phi'$ can have at most one local minimum
and at most one local maximum. Hence $\phi'=0$ has at most three solutions. If there is a unique solution $\hat p =0$ then $\phi'$ is negative in $(0,1]$ whence $\phi$ is maximized at $0$. If there are two solutions: $0$ and $\hat p$, then $\phi'$ is negative in $(0,\hat p)$ and in $(\hat p,1]$ whence $\phi$ is maximized at $0$. If there are three solutions: $0,p_1$ and $\hat p$, then the sign pattern of $\phi'$ is $(-,+,-)$ on the intervals $(0,p_1)$, $(p_1,\hat p)$, $(\hat p,1]$, whence $\phi$ is maximized either at $0$ or at $\hat p$.
\end{proof}

It turns out that, in many cases, deciding whether \(p=0\) or \(p=\hat p\) maximizes
\(\phi\) is precisely the question of whether the \((d+1)\)-core of a corresponding random graph model has edge density smaller or greater than \(d\). In such cases, maximizing \(\phi\) is the same as optimizing the
right-hand side of \eqref{eq:core-rank-upper-bound}.

\begin{figure}[h!]
    \centering
    \includegraphics[width=\textwidth]{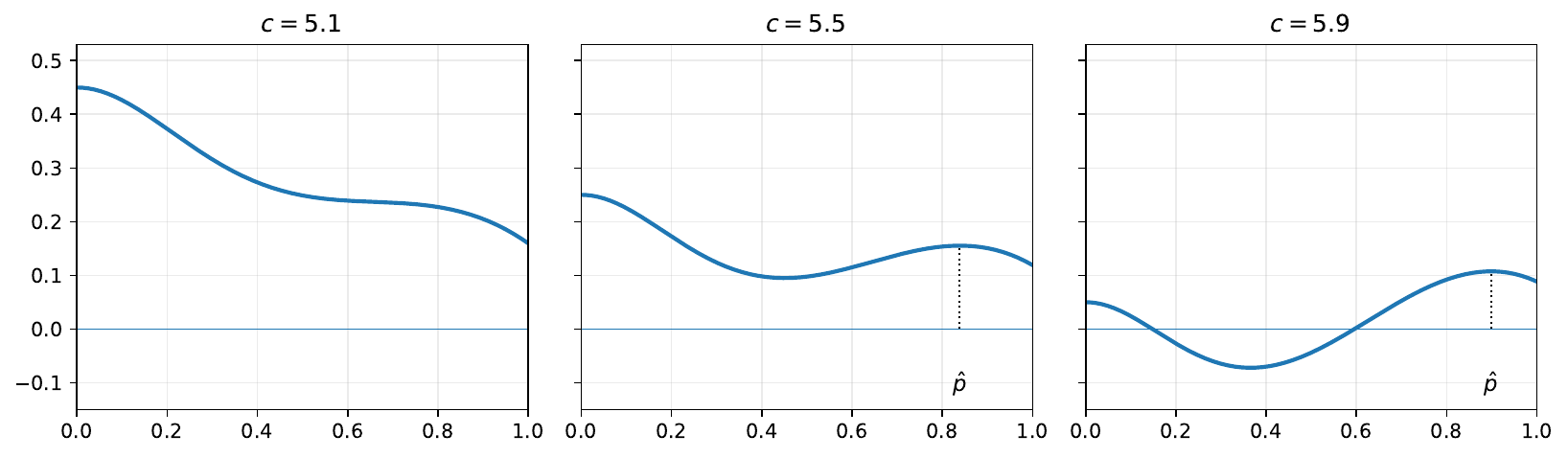}
    \caption{
    The function $\phi(p)$ for $d=3$ and $X=\operatorname{Poi}(c)$, where
    \(c=5.1,5.5,5.9\). On the left, \(c<\gamma_3\), and \(p=0\) is the only
    solution of \(\phi'(p)=0\). In the middle, \(\gamma_3<c<c_3\), two
    additional critical points appear, but the value at \(p=0\) is still larger
    than the value at \(\hat p\). On the right, \(c>c_3\), and the maximum is
    attained at \(\hat p\).
    }
    \label{fig:phi-d3}
\end{figure}

\begin{definition}\label{def:Xtyp}
    A (possibly random) graph sequence $G_n$ with $|V(G_n)|\to\infty$ is called $(X,d)$-typical if the following conditions hold.

\begin{itemize}
    \item The local weak limit of $G_n$ is the $\mathrm{GW}(X,Y)$-distributed rooted tree.
    \item The number $m$ of edges in $G_n$ satisfies that a.a.s. $m=\frac12\E[X]n+o(n).$
    \item Let $n_{d+1},~m_{d+1}$ denote the number of vertices and edges in the $(d+1)$-core of $G_n$. Then, a.a.s.,
    \begin{equation}\label{eq:JLconditions}
        {n_{d+1}}=\P(X_{\hat p}\ge d+1)n+o(n)\,,\,\,
    m_{d+1}=
    \frac{\E[X]\hat p^2}2n+o(n)\,.
    \end{equation}
\end{itemize}
\end{definition}

For example, every essentially high-girth family of \(k\)-regular graphs is
\((k,d)\)-typical. In addition, \(G(n,c/n)\) is
\((\operatorname{Poi}(c),d)\)-typical for every \(c\ne\gamma_d\)
by~\cite{PittelSpencerWormaldCore} (their proof does not apply at
criticality).
More generally, the parameters in \eqref{eq:JLconditions} appear in the work of Janson and Luczak on the $k$-core of random graphs \cite{JansonSimplekcore}. Let $(d_i)_{i\in [n]}$ be non-negative integers with even sum such that $\sum_i\mathbf 1_{d_i=s}=X(s)n+o(n)$ as $n\to\infty$ for every fixed $s$. Let $G_n$ be the configuration-model random graph with degree sequence $(d_i)$: attach $d_i$ half-edges to vertex $i$, and pair all half-edges uniformly at random. It is proved in \cite{JansonSimplekcore} that, under some technical assumptions, $G_n$ is $(X,d)$-typical. 

The next claim relates the two candidate values for maximizing $\phi$ with the parameters of the $(d+1)$-core from \eqref{eq:JLconditions}.
\begin{claim}\label{clm:thisclm}
Let $G_n$ be an $n$-vertex $(X,d)$-typical (possibly random) graph. Then, a.a.s.,
    \begin{enumerate}
    \item \(dn-m=\phi(0)\cdot n+o(n)\), and
    \item \(dn-m-dn_{d+1}+m_{d+1}= \phi(\hat p)\cdot n+o(n)\).
\end{enumerate}
\end{claim}
In particular, $\phi(\hat p)>\phi(0)$ if and only if the asymptotic edge density in the $(d+1)$-core of $G_n$ is greater than $d$.
\begin{proof}
The proof follows a straightforward analysis. Note that by local weak convergence, 
\begin{enumerate}
    \item Since $X_p=0$ if $p=0$, we have that $\phi(0) = d-\E[X]/2$. Hence, the claim follows from the second item in Definition \ref{def:Xtyp}.
    
    \item If $\hat p=0$, then $n_{d+1},m_{d+1}=o(n)$ and the claim follows directly from the first item of this claim. Otherwise, note that $\mathbf 1_{X_p<d}(d-X_p)=d-\min(X_p,d)$, so we derive
    \[
    \phi(\hat p)=d-\E[\mathbf{1}_{X_{\hat p}\le d}X_{\hat p}]-d\cdot\P(X_{\hat p}\ge d+1)-\frac{\E[X]}{2}+\hat p\E[X]-\frac{\hat p^2\E[X]}{2}.
    \]

    We have $\hat p\E[X]=\E[X_{\hat p}]$, hence
    \[
    -\E[\mathbf{1}_{X_{\hat p}\le d}X_{\hat p}]+\hat p\E[X] = 
    \E[\mathbf{1}_{X_{\hat p}\ge d+1}X_{\hat p}]\,.
    \]
    In addition, using the fact that $Y_{\hat p}$ is the size-biasing of $X_{\hat p}$ and \eqref{eq:sizeBiasF}, we find that 
    \begin{equation}\label{eq:simple_hatp}
    \E[\mathbf{1}_{X_{\hat p}\ge d+1}X_{\hat p}]=\E[X_{\hat p}]\P(Y_{\hat p}\ge d) = \hat p^2\E[X]\,,    
    \end{equation}
    where the last equality follows from the fact that $\hat p$ is a solution of \eqref{eq:px-py-bin-recursion}. By combining these equalities we find that
    \[
    \phi(\hat p) = d - \frac{\E[X]}2-d\cdot\P(X_{\hat p}\ge d+1)+\frac{\hat p^2\E[X]}{2}\,,
    \]
    and the claim follows.
\end{enumerate}
\end{proof}

The main theorem of this subsection shows that under the assumptions above, \eqref{eq:core-rank-upper-bound} is asymptotically attained.
\begin{theorem}\label{thm:rd_degreeX}
Suppose that $X$ satisfies Condition \ref{ass:max_phi}, and 
let $G_n$ be an $n$-vertex $(X,d)$-typical (possibly random) graph. Then, a.a.s.,
\[ 
\mathrm{rk}_d(G_n)=
    m+\min\left(0,\,dn_{d+1}-m_{d+1}\right) + o(n)\,.
\]
\end{theorem}
We start with the following easy claim in convergence of random variables.
\begin{claim}
    Let $w$ be a real number, and suppose that $Z_n,W_n$ are sequences of random variables satisfying $\max(0,W_n)\le Z_n$, $W_n\xrightarrow{p}w$ and $\limsup\E Z_n\le w$. Then $Z_n \xrightarrow{p}w$.
\end{claim}
\begin{proof}
    Since \(0\le Z_n\), necessarily \(w\ge 0\). For every \(\varepsilon>0\),
    \[
        \P(Z_n<w-\varepsilon)\le \P(W_n<w-\varepsilon)\to 0 .
    \]
    Also,
    \[
        \E Z_n\ge (w-\varepsilon)\P(W_n\ge w-\varepsilon)
        +\sqrt\varepsilon \P(Z_n\ge w+\sqrt\varepsilon,\ W_n\ge w-\varepsilon).
    \]
    Taking limsup gives
    \[
        w\ge (w-\varepsilon)+\sqrt\varepsilon\P(Z_n\ge w+\sqrt\varepsilon)\,,
    \]
    hence \(Z_n\to w\) in probability.
\end{proof}
\begin{proof}[Proof of Theorem \ref{thm:rd_degreeX}]
Sample $\p_n:V(G_n)\to\R^d$ by drawing each $\p(v)$ from some admissible probability distribution $\nu$ on $\R^d$.
Let $$Z_n= \frac{1}n\dim Z(G_n,\p_n)\ge 0\,,$$ and $$W_n = \frac{1}n\max(dn-m,\,dn-m - dn_{d+1}+m_{d+1})\,.$$
First, since \(\p_n\) is almost surely generic,
\[
\mathrm{rk}_d(G_n)=dn-\dim Z(G_n,\p_n)
\]
holds almost surely. Hence, by
\eqref{eq:core-rank-upper-bound}, we find that $Z_n\ge W_n$ holds almost surely.
Moreover, by Condition \ref{ass:max_phi} and Claim \ref{clm:thisclm}, we have that, as $n\to\infty$,
\begin{equation}\label{eq:W}
W_n\xrightarrow{p} \max\{\phi(p):p\in[0,1]\}\,.  
\end{equation}
Finally, equation \eqref{eq:dimZ_local} implies that
\[
\dim Z(G_n,\p_n) = \sum_{o_n\in V(G_n)}\phi_{G_n,\p_n}(o_n)\,.
\]
In addition, since the local weak limit of $G_n$ is a $\mathrm{GW}(X,Y)$-distributed tree $(T,o)$, we find by Lemma \ref{lem:upper_semi} and Theorem \ref{thm:Exp_loc_flex} that
\[
\limsup \E[Z_n] = \limsup \E[\phi_{G_n,\p_n}(o_n)] \le \E[\phi_{T,\nu}(o)]\le \max\{\phi(p):p\in[0,1]\}\,.
\]
By the previous claim, we find that a.a.s.
\[
\dim Z(G_n,\p_n)
=
\max(dn-m,\,dn-m-(dn_{d+1}-m_{d+1}))+o(n)\,,
\]
and the claim follows.
\end{proof}

We derive an analog corollary on the local flexibility of Galton--Watson trees.

\begin{corollary}\label{cor:treeX}
    Suppose that $X$ satisfies Condition \ref{ass:max_phi} and that there exists a sequence of $(X,d)$-typical graphs. Let $(T,o)$ be a $\mathrm{GW}(X,Y)$-distributed tree and $\nu$ an admissible probability measure on $\R^d$. Then,
    \[
    \E[\phi_{T,\nu}(o)] = d - \frac{\E[X]}{2} + \max\left(0,-d\cdot\P(X_{\hat p}\ge d+1)+\frac{\hat p^2\E[X]}{2}\right)\,.
    \]
\end{corollary}
\begin{proof}
By Theorem \ref{thm:Exp_loc_flex},
\[
    \E[\phi_{T,\nu}(o)]\le \max\{\phi(p):p\in[0,1]\}\,,
\]
and by Condition \ref{ass:max_phi} the maximum is attained at either \(0\) or
\(\hat p\). Additionally, the proof of Claim \ref{clm:thisclm} gives
\[
    \phi(0)=d-\frac{\E[X]}2
\]
and
\[
    \phi(\hat p)
    =
    d-\frac{\E[X]}2
    -d\P(X_{\hat p}\ge d+1)
    +\frac{\hat p^2\E[X]}2 .
\]
It remains to show that the upper bound is attained. Indeed, let
$G_n$ be an $(X,d)$-typical graph sequence, and let $\p_n:V(G_n)\to\R^d$ be sampled according to
$\nu$. By the local weak convergence of \(G_n\) to \(\mathrm{GW}(X,Y)\)
and Lemma \ref{lem:upper_semi}, we have
\[
    \E[\phi_{T,\nu}(o)]\ge \limsup \frac{1}{n}\E\dim Z(G_n,\p_n)=
    \max\{\phi(p):p\in[0,1]\},
\]
where the last equality follows from the proof of Theorem
\ref{thm:rd_degreeX}.
\end{proof}

\subsection{Erd\H{o}s--R\'enyi random graphs}
\label{sec:ER}
We proceed to the proofs of Theorems \ref{thm:mainGnp_rank} and \ref{thm:mainGnp_closure} regarding the $d$-rigidity phase
transition in $G(n,p)$. 

Let $c>0$ be fixed, and let $X= \operatorname{Poi}(c)$. In this case, the shifted size-biasing
$Y$ of $X$ is also $\operatorname{Poi}(c)$. We first note that \(X\) satisfies Condition \ref{ass:max_phi}. Indeed, since \(Y\sim \operatorname{Poi}(c)\), we have
\(Y_p\sim \operatorname{Poi}(cp)\). Hence, using \eqref{eq:phi''}, we find that
\[
\phi''(p)=c\left(\frac{e^{-cp}c^dp^{d-1}}{(d-1)!}-1 \right)\,.
\]
This function is unimodal in $[0,1]$ and the claim follows from Claim \ref{clm:max_phi}.

It is well-known that \(G(n,c/n)\) converges locally weakly to $\mathrm{GW}(X,Y)$, and the parameters of its $(d+1)$-core were found in~\cite{PittelSpencerWormaldCore}, showing it is $(X,d)$-typical for all $c\ne \gamma_d$.

Consequently, the following corollary is a special case of Theorem \ref{thm:rd_degreeX}.
\begin{corollary}\label{cor:rd_ER}
    Let \(c>0\) be fixed, and let \(G\sim G(n,c/n)\). Then, a.a.s.,       
    \begin{itemize}
        \item If \(c<c_d\) then \(\mathrm{rk}_d(G) = (1+o(1))\frac{c}2n.\)
        \item If \(c>c_d\) then \[\mathrm{rk}_d(G) = (1+o(1))\left(\frac c2 + d\,\P\big(\operatorname{Poi}(c\hat p)\ge d+1\big)
        - \frac{c\hat p^2}{2}
        \right)n\,.\]
    \end{itemize}
\end{corollary}
Formally, we are not able to apply Theorem \ref{thm:rd_degreeX} for $c=\gamma_d$, but the statement follows from the monotonicity of the rank.
\begin{proof}[Proof of Theorem \ref{thm:mainGnp_rank}]
The second item is already proved in Corollary \ref{cor:rd_ER}, so we focus on the first item. 
Let \(c<c_d\), and choose
\(c<c'<c_d\). We sample \(G'\sim G(n,c'/n)\), and obtain
\(G\sim G(n,c/n)\) by retaining each edge of \(G'\) independently with
probability \(c/c'\). By Corollary \ref{cor:rd_ER}, we have a.a.s.
\[
    \mathrm{rk}_d(G')=(1+o(1))\frac{c'}2 n.
\]
Therefore, for a generic placement \(\p\), the left-kernel \(S\) of
\(R(G',\p)\) satisfies
\[
    \dim S=m(G')-\mathrm{rk}_d(G')=o(n).
\]

We now expose the retained graph by deleting the edges of \(G'\) one at a
time in a uniformly random order, until each edge has been retained with
probability \(c/c'\). Let \(S_t\) be the left-kernel after \(t\) deletions.
We use the standard gap theorem for \(k\)-cores in random graphs:
There exists $\alpha >0$ such that, a.a.s., every nonempty subgraph of
\(G(n,c'/n)\) with minimum degree at least \(d+1\) has at least
\(\alpha n\) vertices~\cite{LuczakCoreSize}. Hence we may assume that every circuit of the \(d\)-rigidity matroid contained in \(G'\) has at least
\(\alpha n\) edges. Indeed, every such circuit has minimum degree at least
\(d+1\).

Thus, whenever \(S_t\neq 0\), the current graph contains a circuit with at
least \(\alpha n\) edges. The next deleted edge belongs to this circuit
with probability bounded below by a positive constant depending only on
\(c'\) and \(\alpha\). If this happens, then \(\dim S_t\) decreases by
one. Since initially \(\dim S=o(n)\), while the number of deleted edges is a.a.s.
\((1-c/c'+o(1))|E(G')|=\Omega(n)\), a Chernoff bound implies that a.a.s.
the left-kernel vanishes before the deletion process ends. Hence the
resulting graph \(G\sim G(n,c/n)\) is a.a.s. $d$-independent.
\end{proof}

We turn to the phase transition in the rigidity closure.

\begin{proof}[Proof of Theorem \ref{thm:mainGnp_closure}]
We start with item $(1b)$. 
Choose \(c<c'<c_d\), and couple \(G\sim G(n,c/n)\) and \(G'\sim G(n,c'/n)\) by first
sampling \(G\), and then sprinkling each missing edge independently with
probability \((c'-c)/(n-c)\). Suppose that, with probability bounded away
from zero,
\[
    |\operatorname{cl}_d(G)\setminus E(G)|>\varepsilon n
\]
for some \(\varepsilon>0\). Conditional on this event, the sprinkling adds
at least one edge of \(\operatorname{cl}_d(G)\setminus E(G)\) with
probability bounded away from zero, implying that, with probability bounded away from
zero, \(G'\) is not $d$-independent, in contradiction to Theorem \ref{thm:mainGnp_rank}. In addition, it is easy to see that a.a.s., every $v$ vertices in $G$ span at most $O(v)$ edges, 
whence any clique in $\operatorname{cl}_d(G)$ of size $v$ uses at least $\Omega(v^2)$ edges outside of $G$. Hence, such a clique must have size $v=o(\sqrt n)$. This concludes $(1b)$.

We turn to derive $(1a)$ for $d\ge 3$ since the case $d=2$ was proved in \cite{KasiviswanathanMooreTheranRigidityTransition}. 
Every induced $d$-rigid subgraph spans a clique in the closure and therefore,
by item $(1b)$, is of size $o(\sqrt n)$. In addition, it must have minimum degree at least $3$, in contradiction to the gap theorem for \(k\)-cores \cite{LuczakCoreSize}.

\smallskip

Turning to item $(2)$, we start by giving a lower bound on the number of edges in $\operatorname{cl}_d(G)$.
\begin{claim}\label{clm:closure_edges_lower}
    A.a.s.,
\[
    \left|\operatorname{cl}_d(G)\right|
    \ge
    \left(\hat p^2-o(1)\right)\binom{n}2\,.
\]
\end{claim}

\begin{proof}
Fix \(\varepsilon>0\). Choose \(c'\) with \(c_d<c'<c\) sufficiently close
to \(c\) so that
\begin{equation}\label{eq:cprime}
    \frac{\rho(c)-\rho(c')}{(c-c')}
    <
    \frac{1-\hat p^2}{2}+\varepsilon,    
\end{equation}
where
\[
    \rho(c)
    =
    \frac{c}{2}
    -
    \frac{c\hat p(c)^2}{2}
    +
    d\,\P\big(\operatorname{Poi}(c\hat p(c))\ge d+1\big)
\]
is the limiting normalized rank from Corollary \ref{cor:rd_ER}. This is
possible since
\[
    \rho'(c)=\frac{1-\hat p^2}{2}.
\]
Indeed,
\[
\rho'(c)
=
\frac12-\frac{\hat p^2}{2}
-c\hat p\hat p'
+
d\,\P(\operatorname{Poi}(c\hat p)=d)(\hat p+c\hat p').
\]
The last two terms cancel since
\[
d\,\P(\operatorname{Poi}(c\hat p)=d)
=
c\hat p\,\P(\operatorname{Poi}(c\hat p)=d-1),
\]
and 
\[
\hat p'
=
\P(\operatorname{Poi}(c\hat p)=d-1)(\hat p+c\hat p')
\]
by differentiating the fixed-point equation.

Couple the random graph process so that \(G'\sim G(n,c'/n)\) and
\(G\sim G(n,c/n)\), with \(G'\subseteq G\). Equivalently, after
sampling \(G'\), we sprinkle \(T\) additional random edges, where a.a.s.
\begin{equation}\label{eq:T}
    T=(1+o(1))\frac{c-c'}{2}n.    
\end{equation}
Suppose, for contradiction, that with probability bounded away from zero,
\[
    |\operatorname{cl}_d(G)|
    \le
    \left(\hat p^2-4\varepsilon\right)\binom n2 .
\]

Let
\[
    G'=G^0\subseteq G^1\subseteq\cdots\subseteq G^T=G
\]
be the graphs obtained during the sprinkling process, where \(G^{i+1}\)
is obtained from \(G^i\) by adding one uniformly chosen edge from
\(E(K_n)\setminus E(G^i)\). By monotonicity of matroid closure, on the
event above we have, for every \(0\le i\le T\),
\[
    \operatorname{cl}_d(G^i)\subseteq \operatorname{cl}_d(G),
\]
and therefore
\[
    |\operatorname{cl}_d(G^i)|
    \le
    \left(\hat p^2-4\varepsilon\right)\binom n2 .
\]

Now expose the next edge as follows. Given \(G^i\), let
\[
    a_i
    =
    \frac{\binom n2-|\operatorname{cl}_d(G^i)|}
         {\binom n2-|E(G^i)|} 
         \ge 1-\hat p^2+4\varepsilon\,.
\]
Draw \(r_i\sim U[0,1]\). If \(r_i\le a_i\), choose the next edge uniformly
from
\[
    E(K_n)\setminus \operatorname{cl}_d(G^i),
\]
and hence the rank increases by one. Otherwise, choose it uniformly from
\[
    \operatorname{cl}_d(G^i)\setminus E(G^i).
\]
This gives exactly a uniformly chosen edge from
\(E(K_n)\setminus E(G^i)\). 
Therefore the number of rank-increasing sprinkled edges is at least
\[
    \#\{i:r_i\le 1-\hat p^2+4\varepsilon+o(1)\}.
\]
By the law of large numbers and \eqref{eq:T}, this number is a.a.s. at least
    $(1-\hat p^2+3\varepsilon)T.$
    
On the other hand, by Corollary \ref{cor:rd_ER}, a.a.s. the actual rank
increase from \(G'\) to \(G\) is
\[
    \big(\rho(c)-\rho(c')+o(1)\big)n <(1+o(1))\cdot( 1-\hat p^2+2\varepsilon)T
\]
due to \eqref{eq:cprime} and \eqref{eq:T},
yielding a contradiction. Since \(\varepsilon>0\) was arbitrary, the claim follows.
\end{proof}

Denote by $\widehat V_r$ the set of vertices $v$ in $G$ whose $r$-neighborhood is a tree that contains a $d$-ary tree of height $r$ rooted at $v$, and by $C_r$ the set of vertices $v$ whose $r$-neighborhood is not a tree.

\begin{claim}\label{clm:deterministic_closure_r}
Let $G$ be a graph, $r\ge 2$ an integer and $uv\in\operatorname{cl}_d(G)$. Then, either
    $\{u,v\}\subset \widehat V_r$,
    $\operatorname{dist}_G(u,v)<r$, or
    $\{u,v\}\cap C_r\ne\emptyset$.
\end{claim}
\begin{proof}
    We show that if the $r$-neighborhood
of $u$ is a tree which does not contain a rooted $d$-ary tree of height
$r$, and if $\operatorname{dist}_G(u,v)\ge r$, then
$uv\notin \operatorname{cl}_d(G)$.
Indeed, denote by $(T,u)$ the $r$-neighborhood of $u$, rooted at $u$.
Embed $G$ in $\mathbb R^d$ according to an admissible measure $\nu$.
Associate a $d\times d$ matrix $Q_x$ to every vertex $x$ of $T$ as
follows. If $x$ has height $r$ in $T$, set $Q_x=0$ (which means that these boundary vertices are pinned); otherwise compute
$Q_x$ recursively, using the sampled embedding $\p$ and
Theorem~\ref{thm:recursion}. Since $(T,u)$ does not contain a rooted
$d$-ary subtree of height $r$, the recursion gives $Q_u\ne 0$.
In particular, there is a non-trivial infinitesimal motion $q$ of $T$
which preserves all edge lengths in $T$, satisfies $q(x)=0$ for every
vertex $x$ at height $r$, and satisfies $q(u)\ne 0$. Extend $q$ to all
of $G$ by setting $q(w)=0$ for every vertex $w\notin T$. Since the boundary
vertices of $T$ are fixed, this extension preserves the length of every
edge of $G$, hence $q\in\ker L(G,\p)$.
In particular, $q(v)=0$ while $q(u)\ne 0$. For a generic admissible
placement, the vector $q(u)-q(v)$ is not orthogonal to
$\p(u)-\p(v)$. Thus adding the edge $uv$ to $G$ increases its $d$-rigidity rank, and therefore $uv\notin \operatorname{cl}_d(G)$.
\end{proof}

Since \(G\sim G(n,c/n)\) converges locally weakly to
\(\mathrm{GW}(\operatorname{Poi}(c))\), we a.a.s. have $|C_r|=o(n)$ for every fixed \(r\ge 2\). 
In addition, the expected number of pairs $uv$ of distance less than $r$ in $G$ is $O(n)$ and is therefore a.a.s. of size $o(n^2)$. Therefore,
    Claim \ref{clm:deterministic_closure_r} gives that for every fixed $r\ge 2$, a.a.s.
    \begin{equation}\label{eq:closure-Vr-upper-bound}
    \left|
    \operatorname{cl}_d(G)\setminus \binom{\widehat V_r}{2}
    \right| =o(n^2)\,.
\end{equation}

In addition, \(\E[|\widehat V_r|]/n\to q_r\), where \(q_r\) is the
probability that the first \(r\) generations of
\(\mathrm{GW}(\operatorname{Poi}(c))\) contain a complete \(d\)-ary
tree rooted at the root. Moreover, a standard second moment argument
applies: for fixed \(r\), the \(r\)-neighborhoods of two uniformly chosen
vertices are disjoint and asymptotically independent with probability
\(1-o(1)\). Hence a.a.s. 
\begin{equation}\label{eq:hatVr}
    |\widehat V_r|=(1+o(1))q_rn\,.
\end{equation}
In addition, the numbers \(q_r\) satisfy the recursion
\[
    q_{r+1}
    =
    \P(\operatorname{Poi}(c q_r)\ge d),
    \qquad q_0=1\,,
\]
whence \(q_r\) converges to $\hat p.$ 
Therefore, by taking $r\to\infty$, 
\eqref{eq:closure-Vr-upper-bound}
and \eqref{eq:hatVr} imply that a.a.s.,
\begin{equation}\label{eq:closure_edges_upper}
\left|\operatorname{cl}_d(G)\right|
    \le
    \left(\hat p^2+o(1)\right)\binom{n}2\,,
\end{equation}
matching the lower bound from \ref{clm:closure_edges_lower}.

\medskip 

Next, we find a large clique in the closure.
Let $\varepsilon>0$. 
By Claim \ref{clm:closure_edges_lower}, 
\eqref{eq:closure-Vr-upper-bound} and 
\eqref{eq:hatVr} it holds that for all sufficiently large $r=r(\varepsilon)$, a.a.s. the subgraph of $\operatorname{cl}_d(G)$ induced by $\widehat V_r$ has edge density at least $1-\varepsilon^2$. By a standard counting argument, there exists a subset
\(A\subseteq \widehat V_r\) with
\[
    |A|\ge (1-\varepsilon)(|\widehat V_r|-1)\,,
\]
such that every $v\in A$ has at least $(1-\varepsilon)(|\widehat V_r|-1)$ neighbors from $\widehat V_r$ in $\operatorname{cl}_d(G)$. Hence, the subgraph $H$ of $\operatorname{cl}_d(G)$ induced by $A$ has minimum degree at least $(1-2\varepsilon)(|\widehat V_r|-1)$.
For a sufficiently small $\varepsilon>0$, the minimum-degree condition for rigidity from \cite{KrivelevichLewMichaeliRigidity} holds, implying that the graph \(H\) is generically \(d\)-rigid. Since \(H\) is an induced subgraph
of \(\operatorname{cl}_d(G)\), it follows that \(A\) induces a clique in
\(\operatorname{cl}_d(G)\). Letting $\varepsilon \to 0$ and $r=r(\varepsilon)\to\infty$ shows that $\operatorname{cl}_d(G)$ a.a.s. contains a clique of size $(\hat p-o(1))n$. Note that no clique in the closure has a strictly larger density due to \eqref{eq:closure_edges_upper}.

\medskip

Let \(\widehat G\) denote the \(((d+1)+d)\)-core of \(G\),
obtained from the
$(d+1)$-core by repeatedly adding vertices with at least $d$ neighbors in the already constructed graph.
It remains to show that the largest clique in the closure contains almost all the vertices of $\widehat G$. We claim that 
$$V(\widehat G)\subset \widehat V_r \cup C_r\,.$$
Indeed, suppose that $v\in V(\widehat G)\setminus C_r$. If $v$ is in the $(d+1)$-core and its $r$-neighborhood is
acyclic, then it contains the first $r$ layers of the $(d+1)$-regular
tree, whence $v\in \widehat V_r$. Otherwise, fix an order in which the vertices of
\(\widehat G\setminus V(G_{d+1})\) are added. Since \(v\) is not in the
\((d+1)\)-core, the \((d+1)\)-core is unchanged after deleting \(v\).
At the moment \(v\) is added, it has at least \(d\) neighbors that were
already present; by using the same earlier part of the addition order,
these neighbors belong to the \(((d+1)+d)\)-core of \(G\setminus v\).
Moreover, since \(v\notin C_r\), their \((r-1)\)-neighborhoods in
\(G\setminus v\) are trees. We then proceed by induction on \(r\) to show
that \(v\in \widehat V_r\).

Let \(A\) be a largest clique in \(\operatorname{cl}_d(G)\). By the previous paragraph, a.a.s.
\[
    |A|=(\hat p+o(1))n.
\]
Moreover, by \eqref{eq:closure-Vr-upper-bound}, a.a.s. \(A\) has \(o(n)\) vertices outside
\(\widehat V_r\).
Therefore, using \eqref{eq:hatVr},
\[
    |\widehat V_r\setminus A|
    =
    (q_r-\hat p+o(1))n.
\]
Since \(q_r\to \hat p\), for every \(\varepsilon>0\) and all sufficiently large fixed \(r\), a.a.s.
\[
|V(\widehat G)\setminus A|
\le
|\widehat V_r\setminus A|+|C_r|
\le
(\varepsilon+o(1))n .
\]
The proof is concluded by taking \(\varepsilon\to 0\).

\end{proof}

\subsection{Essentially high-girth regular graphs}
\label{sec:regular}
We now consider the special case where \(X\equiv k\), and the shifted
size-biasing of \(X\) is \(Y\equiv k-1\).
Therefore, by \eqref{eq:phi''},
\[
\phi''(p)=k\left(
d \binom{k-1}{d}p^{d-1}(1-p)^{k-1-d}-1
\right)\,.
\]
For \(k\ge d+1\), the function $\phi''$ is strictly unimodal, hence Condition \ref{ass:max_phi} holds by Claim \ref{clm:max_phi}. If \(k\le d\), then 
$\phi(p)=d-\frac{k}{2}(1+p^2)$ attains its maximum at $p=0$, and Condition \ref{ass:max_phi} holds.

\begin{proof}[Proof of Theorem~\ref{thm:main_high_girth_regular}]
An essentially high-girth sequence of \(k\)-regular graphs is $(X,d)$-typical. Indeed, local weak convergence holds by definition of essential high-girth, and $m=kn/2$. In addition, the equation
$p=\P(Y_p\ge d) =\P(\mathrm{Bin}(k-1,p)\ge d)$ has a unique solution $\hat p=0$ if $k\le d$, and the largest solution in $[0,1]$ is $\hat p=1$ if $k\ge d+1$. In the former case $n_{d+1}=m_{d+1}=0$ and in the latter case $n_{d+1}=n,~m_{d+1}=m$, thus \eqref{eq:JLconditions} is satisfied. The proof is completed by applying Theorem \ref{thm:rd_degreeX}.
\end{proof}

\begin{proof}[Proof of Theorem~\ref{thm:Tk}]
Since essentially high-girth $k$-regular graphs exist, we can apply Corollary~\ref{cor:treeX} for $\bbT_k$ and derive
\[
 \E[\phi_{\bbT_k,\nu}(o)]=\max(d-k/2,0)\,.
\]
By Lemma~\ref{lem:Qo0}, the event \(Q_o=0\) admits a zero-one law. Hence
\(Q_o=0\) almost surely if and only if
\(
    \E[\phi_{\bbT_k,\nu}(o)]=0,
\)
which holds precisely when \(k\ge 2d\).
In addition, \(Q_o=0\) almost surely if and only if \(Q_v=0\) almost surely
for every vertex \(v\). The claim follows.
\end{proof}

\section{Discussion}\label{sec:discussion}
Many fascinating open problems remain.
\begin{enumerate}
    \item The $d$-rigidity phase transition in $G(n,p)$ is now established, but both the subcritical and supercritical regimes should be better understood.
    Does it hold that:
    \begin{itemize}
        \item If $c<c_d$ then $G(n,c/n)$ is a.a.s. closed in the $d$-rigidity matroid? 
        \item If $c>c_d$ then the $((d+1)+d)$-core a.a.s. induces a $d$-rigid subgraph?
    \end{itemize}
    These are the strongest possible statements in the respective regimes, and numerical simulations suggest that they hold.
    \item Similarly,
    \begin{itemize}
    \item Is a random \(k\)-regular graph independent and closed in the generic \(d\)-rigidity matroid for \(k<2d\)?
    \item Is it a.a.s. \(d\)-rigid for \(k\ge 2d\)? 
\end{itemize}
In this paper we only compute its \(d\)-rigidity rank asymptotically, and
prove the corresponding result for strict rigidity of the infinite \(k\)-regular tree.
\item Another natural direction is to extend the results of this work to global
rigidity. For instance, proving that for \(c>c_d\), the \((d+1)\)-core of
\(G(n,c/n)\), or some large subgraph of it, is a.a.s. globally \(d\)-rigid
would resolve the conjecture of
\cite{BernsteinDewarGortlerNixonSitharamTheranMLT} concerning the MLT of
\(G(n,c/n)\)~\cite{THERAN_PRIVATE}.

    \item It would be interesting to generalize the methods of this paper to
other abstract \(d\)-rigidity matroids. The \(d\)-hyperconnectivity matroid
and the cofactor matroid are natural candidates, as both have concrete
matrix descriptions to which the spectral methods developed here may
transfer.
    \item The fact that generic rigidity is a combinatorial property of the graph has
played a huge role in the development of combinatorial rigidity theory. Is
there a similar notion for infinite self-adjoint frameworks?
For instance, we know by Lemma \ref{lem:Qo0} that in a $\nu$-embedding of a rooted tree
\((T,o)\), the event \(Q_o=0\) admits a zero-one law. Does the same hold for all locally
finite graphs \(G\)? Moreover, even for trees, we do not know whether the answer can depend on \(\nu\).
\end{enumerate}

\noindent \textbf{Acknowledgments}
The author thanks Paul Duncan, Alan Lew, Eran Nevo, Orit Raz and Louis Theran for many useful discussions at various stages of this work.

\bibliographystyle{amsplain}
\bibliography{rigid}

@article{ErdosRenyi1959,
  author  = {Erd{\H{o}}s, Paul and R{\'e}nyi, Alfr{\'e}d},
  title   = {On Random Graphs. {I}},
  journal = {Publicationes Mathematicae Debrecen},
  volume  = {6},
  pages   = {290--297},
  year    = {1959}
}

@article{ErdosRenyi1960,
  author  = {Erd{\H{o}}s, Paul and R{\'e}nyi, Alfr{\'e}d},
  title   = {On the Evolution of Random Graphs},
  journal = {Publications of the Mathematical Institute of the Hungarian Academy of Sciences},
  volume  = {5},
  pages   = {17--61},
  year    = {1960}
}

@book{BollobasRandomGraphs,
  author    = {Bollob{\'a}s, B{\'e}la},
  title     = {Random Graphs},
  edition   = {2},
  series    = {Cambridge Studies in Advanced Mathematics},
  volume    = {73},
  publisher = {Cambridge University Press},
  address   = {Cambridge},
  year      = {2001},
  doi       = {10.1017/CBO9780511814068}
}

@book{JansonLuczakRucinskiRandomGraphs,
  author    = {Janson, Svante and Luczak, Tomasz and Ruci{\'n}ski, Andrzej},
  title     = {Random Graphs},
  series    = {Wiley Series in Discrete Mathematics and Optimization},
  publisher = {Wiley},
  address   = {New York},
  year      = {2000},
  doi       = {10.1002/9781118032718}
}

@book{FriezeBook,
  author    = {Frieze, Alan and Karo{\'n}ski, Micha{\l}},
  title     = {Introduction to Random Graphs},
  publisher = {Cambridge University Press},
  address   = {Cambridge},
  year      = {2015},
  doi       = {10.1017/CBO9781316339831}
}

@article{Laman1970,
  author  = {Laman, Gerard},
  title   = {On Graphs and Rigidity of Plane Skeletal Structures},
  journal = {Journal of Engineering Mathematics},
  volume  = {4},
  number  = {4},
  pages   = {331--340},
  year    = {1970},
  doi     = {10.1007/BF01534980}
}

@article{AsimowRoth1978,
  author  = {Asimow, Leonard and Roth, Ben},
  title   = {The Rigidity of Graphs},
  journal = {Transactions of the American Mathematical Society},
  volume  = {245},
  pages   = {279--289},
  year    = {1978},
  doi     = {10.1090/S0002-9947-1978-0511410-9}
}

@article{AsimowRoth1979,
  author  = {Asimow, Leonard and Roth, Ben},
  title   = {The Rigidity of Graphs. {II}},
  journal = {Journal of Mathematical Analysis and Applications},
  volume  = {68},
  number  = {1},
  pages   = {171--190},
  year    = {1979},
  doi     = {10.1016/0022-247X(79)90108-2}
}

@book{GraverServatiusServatius1993,
  author    = {Graver, Jack and Servatius, Brigitte and Servatius, Herman},
  title     = {Combinatorial Rigidity},
  series    = {Graduate Studies in Mathematics},
  volume    = {2},
  publisher = {American Mathematical Society},
  address   = {Providence, RI},
  year      = {1993},
  doi       = {10.1090/gsm/002}
}

@article{BenjaminiSchramm2001,
  author  = {Benjamini, Itai and Schramm, Oded},
  title   = {Recurrence of Distributional Limits of Finite Planar Graphs},
  journal = {Electronic Journal of Probability},
  volume  = {6},
  pages   = {1--13},
  year    = {2001},
  doi     = {10.1214/EJP.v6-96}
}

@incollection{AldousSteeleObjectiveMethod,
  author    = {Aldous, David and Steele, J. Michael},
  title     = {The Objective Method: Probabilistic Combinatorial Optimization and Local Weak Convergence},
  booktitle = {Probability on Discrete Structures},
  editor    = {Kesten, Harry},
  series    = {Encyclopaedia of Mathematical Sciences},
  volume    = {110},
  pages     = {1--72},
  publisher = {Springer},
  address   = {Berlin},
  year      = {2004},
  doi       = {10.1007/978-3-662-09444-0_1}
}

@article{PittelSpencerWormaldCore,
  author  = {Pittel, Boris and Spencer, Joel and Wormald, Nicholas},
  title   = {Sudden Emergence of a Giant {$k$}-Core in a Random Graph},
  journal = {Journal of Combinatorial Theory, Series B},
  volume  = {67},
  number  = {1},
  pages   = {111--151},
  year    = {1996},
  doi     = {10.1006/jctb.1996.0036}
}

@article{JansonSimplekcore,
  author  = {Janson, Svante and Luczak, Malwina J.},
  title   = {A Simple Solution to the {$k$}-Core Problem},
  journal = {Random Structures \& Algorithms},
  volume  = {30},
  number  = {1--2},
  pages   = {50--62},
  year    = {2007},
  doi     = {10.1002/rsa.20147}
}

@article{LuczakCoreSize,
  author  = {Luczak, Tomasz},
  title   = {Size and Connectivity of the {$k$}-Core of a Random Graph},
  journal = {Discrete Mathematics},
  volume  = {91},
  number  = {1},
  pages   = {61--68},
  year    = {1991},
  doi     = {10.1016/0012-365X(91)90162-U}
}

@inproceedings{CainSandersWormald,
  author    = {Cain, J. A. and Sanders, P. and Wormald, N. C.},
  title     = {The Random Graph Threshold for {$k$}-Orientability and a Fast Algorithm for Optimal Multiple-Choice Allocation},
  booktitle = {Proceedings of the Eighteenth Annual ACM-SIAM Symposium on Discrete Algorithms},
  pages     = {469--476},
  publisher = {SIAM},
  year      = {2007}
}

@inproceedings{FernholzRamachandran,
  author    = {Fernholz, Daniel and Ramachandran, Vijaya},
  title     = {The {$k$}-Orientability Thresholds for {$G_{n,p}$}},
  booktitle = {Proceedings of the Eighteenth Annual ACM-SIAM Symposium on Discrete Algorithms},
  pages     = {459--468},
  publisher = {SIAM},
  year      = {2007}
}

@article{BernsteinDewarGortlerNixonSitharamTheranMLT,
  author  = {Bernstein, Daniel Irving and Dewar, Sean and Gortler, Steven J. and Nixon, Anthony and Sitharam, Meera and Theran, Louis},
  title   = {Maximum likelihood thresholds via graph rigidity},
  journal = {The Annals of Applied Probability},
  volume  = {34},
  number  = {3},
  pages   = {3288--3319},
  year    = {2024},
  doi     = {10.1214/23-AAP2039}
}

@article{Uhler2012,
  author  = {Uhler, Caroline},
  title   = {Geometry of Maximum Likelihood Estimation in Gaussian Graphical Models},
  journal = {The Annals of Statistics},
  volume  = {40},
  number  = {1},
  pages   = {238--261},
  year    = {2012},
  doi     = {10.1214/11-AOS957}
}

@article{GrossSullivant2018,
  author  = {Gross, Elizabeth and Sullivant, Seth},
  title   = {The Maximum Likelihood Threshold of a Graph},
  journal = {Bernoulli},
  volume  = {24},
  number  = {1},
  pages   = {386--407},
  year    = {2018},
  doi     = {10.3150/16-BEJ881}
}

@article{BordenaveLelarge2010,
  author  = {Bordenave, Charles and Lelarge, Marc},
  title   = {Resolvent of Large Random Graphs},
  journal = {Random Structures \& Algorithms},
  volume  = {37},
  number  = {3},
  pages   = {332--352},
  year    = {2010},
  doi     = {10.1002/rsa.20313}
}

@article{BordenaveLelargeSalezDiluted,
  author  = {Bordenave, Charles and Lelarge, Marc and Salez, Justin},
  title   = {The Rank of Diluted Random Graphs},
  journal = {The Annals of Probability},
  volume  = {39},
  number  = {3},
  pages   = {1097--1121},
  year    = {2011},
  doi     = {10.1214/10-AOP567}
}

@article{LinialPeledHomology,
  author  = {Linial, Nathan and Peled, Yuval},
  title   = {On the Phase Transition in Random Simplicial Complexes},
  journal = {Annals of Mathematics},
  volume  = {184},
  number  = {3},
  pages   = {745--773},
  year    = {2016},
  doi     = {10.4007/annals.2016.184.3.3}
}

@unpublished{AbertThomViragSpectralMeasure,
  author  = {Ab{\'e}rt, Mikl{\'o}s and Thom, Andreas and Vir{\'a}g, B{\'a}lint},
  title  = {Benjamini--Schramm convergence and pointwise convergence of the spectral measure},
  year   = {2013},
  note   = {Preprint}
}

@book{ReedSimonI,
  author    = {Reed, Michael and Simon, Barry},
  title     = {Methods of Modern Mathematical Physics. {I}: Functional Analysis},
  publisher = {Academic Press},
  address   = {New York},
  edition   = {Revised and enlarged},
  year      = {1980}
}

@article{KrivelevichLewMichaeliRigidity,
  author  = {Krivelevich, Michael and Lew, Alan and Michaeli, Peleg},
  title   = {Minimum Degree Conditions for Graph Rigidity},
  journal = {Bulletin of the London Mathematical Society},
  volume  = {58},
  number  = {1},
  pages   = {e70279},
  year    = {2026},
  doi     = {10.1112/blms.70279}
}

@article{JSS-planeThreshold,
  author  = {Jackson, Bill and Servatius, Brigitte and Servatius, Herman},
  title   = {The 2-dimensional rigidity of certain families of graphs},
  journal = {Journal of Graph Theory},
  volume  = {54},
  number  = {2},
  pages   = {154--166},
  year    = {2007},
  doi     = {10.1002/jgt.20196}
}

@article{LewNevoPeledRazSharpThreshold,
  author  = {Lew, Alan and Nevo, Eran and Peled, Yuval and Raz, Orit E.},
  title   = {Sharp threshold for rigidity of random graphs},
  journal = {Bulletin of the London Mathematical Society},
  volume  = {55},
  number  = {1},
  pages   = {490--501},
  year    = {2023},
  doi     = {10.1112/blms.12740}
}

@misc{KiralyTheranCoherence,
  author        = {Kir{\'a}ly, Franz J. and Theran, Louis},
  title         = {Coherence and sufficient sampling densities for reconstruction in compressed sensing},
  eprint        = {1302.2767},
  archivePrefix = {arXiv},
  primaryClass  = {cs.LG},
  year          = {2013},
  doi           = {10.48550/arXiv.1302.2767},
  note          = {arXiv:1302.2767}
}

@article{JordanTanigawaRigidityRandomSubgraphs,
  author  = {Jord{\'a}n, Tibor and Tanigawa, Shin-ichi},
  title   = {Rigidity of random subgraphs and eigenvalues of stiffness matrices},
  journal = {SIAM Journal on Discrete Mathematics},
  volume  = {36},
  number  = {3},
  pages   = {2367--2392},
  year    = {2022},
  doi     = {10.1137/20M1349849}
}

@inproceedings{TheranRigidComponentsRandomGraphs,
  author    = {Theran, Louis},
  title     = {Rigid components of random graphs},
  booktitle = {Proceedings of the 21st Canadian Conference on Computational Geometry},
  pages     = {63--66},
  year      = {2009}
}

@inproceedings{KasiviswanathanMooreTheranRigidityTransition,
  author    = {Kasiviswanathan, Shiva Prasad and Moore, Cristopher and Theran, Louis},
  title     = {The rigidity transition in random graphs},
  booktitle = {Proceedings of the Twenty-Second Annual ACM-SIAM Symposium on Discrete Algorithms},
  pages     = {1237--1252},
  publisher = {SIAM},
  year      = {2011},
  doi       = {10.1137/1.9781611973082.94}
}

@article{BarreLelargeMitscheSliders,
  author  = {Barr{\'e}, Julien and Lelarge, Marc and Mitsche, Dieter},
  title   = {On rigidity, orientability and cores of random graphs with sliders},
  journal = {Random Structures \& Algorithms},
  volume  = {52},
  number  = {3},
  pages   = {419--453},
  year    = {2018},
  doi     = {10.1002/rsa.20749}
}

@article{KrivelevichLewMichaeliRigidPartitions,
  author  = {Krivelevich, Michael and Lew, Alan and Michaeli, Peleg},
  title   = {Rigid partitions: From high connectivity to random graphs},
  journal = {Journal of Combinatorial Theory, Series B},
  volume  = {175},
  pages   = {126--170},
  year    = {2025},
  doi     = {10.1016/j.jctb.2025.07.001}
}

@misc{KrivelevichCombinatorialSufficientConditions,
  author        = {Krivelevich, Michael and Lew, Alan and Michaeli, Peleg},
  title         = {Combinatorial Sufficient Conditions for Graph Rigidity and Applications to Random Graphs},
  year          = {2026},
  eprint        = {2602.23713},
  archivePrefix = {arXiv},
  primaryClass  = {math.CO},
  doi           = {10.48550/arXiv.2602.23713},
  note          = {arXiv:2602.23713}
}

@misc{PeledPelegHighDimensional,
  author        = {Peled, Yuval and Peleg, Niv},
  title         = {On the Rigidity of Random Graphs in high-dimensional spaces},
  eprint        = {2412.13127},
  archivePrefix = {arXiv},
  primaryClass  = {math.CO},
  year          = {2024},
  doi           = {10.48550/arXiv.2412.13127},
  note          = {arXiv:2412.13127}
}

@article{GiraoIllingworthMichelPowierskiScottReconstructing,
  author  = {Gir{\~a}o, Ant{\'o}nio and Illingworth, Freddie and Michel, Lukas and Powierski, Emil and Scott, Alex},
  title   = {Reconstructing a point set from a random subset of its pairwise distances},
  journal = {SIAM Journal on Discrete Mathematics},
  volume  = {38},
  number  = {4},
  pages   = {2709--2720},
  year    = {2024},
  doi     = {10.1137/23M1586860}
}

@article{BenjaminiTzalikDetermining,
  title={Determining a Points Configuration from a Subset of the Pairwise Distances},
  author={Benjamini, Itai and Tzalik, Elad},
  journal={arXiv preprint arXiv:2208.13855},
  year={2022}
}

@incollection{MontgomeryNenadovPortierSzaboGlobalRigidityR,
  author    = {Montgomery, Richard and Nenadov, Rajko and Szab{\'o}, Tibor},
  title     = {Global Rigidity of Random Graphs in {$\mathbb{R}$}},
  booktitle = {2023 MATRIX Annals},
  series    = {MATRIX Book Series},
  volume    = {6},
  pages     = {717--724},
  publisher = {Springer},
  year      = {2025},
  doi       = {10.1007/978-3-031-76738-8_47}
}

@misc{ClinchHaslegraveHuynhNixonNAC,
  author        = {Clinch, Katie and Haslegrave, John and Huynh, Tony and Nixon, Anthony},
  title         = {Sharp thresholds for {NAC}-colourings and stable cuts in random graphs},
  eprint        = {2510.05838},
  archivePrefix = {arXiv},
  primaryClass  = {math.CO},
  year          = {2025},
  doi           = {10.48550/arXiv.2510.05838},
  note          = {arXiv:2510.05838}
}

@misc{JordanLiuVillanyiDegreeSum,
  author        = {Jord{\'a}n, Tibor and Liu, Xuemei and Vill{\'a}nyi, Soma},
  title         = {Degree Sum Conditions for Graph Rigidity},
  eprint        = {2510.25689},
  archivePrefix = {arXiv},
  primaryClass  = {math.CO},
  year          = {2025},
  doi           = {10.48550/arXiv.2510.25689},
  note          = {arXiv:2510.25689}
}

@article{KitsonPowerRigidityInfiniteGraphs,
  author  = {Kitson, Derek and Power, Stephen C.},
  title   = {The rigidity of infinite graphs},
  journal = {Discrete \& Computational Geometry},
  volume  = {60},
  number  = {3},
  pages   = {531--557},
  year    = {2018},
  doi     = {10.1007/s00454-018-9993-0}
}

@article{OwenPowerInfiniteFrameworksOperatorTheory,
  author  = {Owen, John C. and Power, Stephen C.},
  title   = {Infinite bar-joint frameworks, crystals and operator theory},
  journal = {New York Journal of Mathematics},
  volume  = {17},
  pages   = {445--490},
  year    = {2011}
}

@misc{THERAN_PRIVATE,
  author = {Theran, Louis},
  title  = {Private communication},
  year   = {2026},
  note   = {Private communication}
}

@article{Thorpe1985,
  author  = {Thorpe, M. F.},
  title   = {Rigidity percolation in glassy structures},
  journal = {Journal of Non-Crystalline Solids},
  volume  = {76},
  number  = {1},
  pages   = {109--116},
  year    = {1985},
  doi     = {10.1016/0022-3093(85)90056-0}
}

@article{Phillips1979,
  author  = {Phillips, J. C.},
  title   = {Topology of covalent non-crystalline solids {I}: Short-range order in chalcogenide alloys},
  journal = {Journal of Non-Crystalline Solids},
  volume  = {34},
  number  = {2},
  pages   = {153--181},
  year    = {1979},
  doi     = {10.1016/0022-3093(79)90033-4}
}

@article{Thorpe1983,
  author  = {Thorpe, M. F.},
  title   = {Continuous deformations in random networks},
  journal = {Journal of Non-Crystalline Solids},
  volume  = {57},
  number  = {3},
  pages   = {355--370},
  year    = {1983},
  doi     = {10.1016/0022-3093(83)90424-6}
}

@article{JacobsThorpe1995,
  author  = {Jacobs, D. J. and Thorpe, M. F.},
  title   = {Generic rigidity percolation: The pebble game},
  journal = {Physical Review Letters},
  volume  = {75},
  number  = {22},
  pages   = {4051--4054},
  year    = {1995},
  doi     = {10.1103/PhysRevLett.75.4051}
}

\end{document}